\documentclass[10pt]{amsart}

\usepackage[all]{xy}
\usepackage{amsmath, amssymb, hyperref}
\usepackage{tikz}

\newtheorem{thm}{Theorem}[section]
\newtheorem{prop}[thm]{Proposition}
\newtheorem{cor}[thm]{Corollary}
\newtheorem{lem}[thm]{Lemma}

\theoremstyle{remark}
\newtheorem{remark}[thm]{Remark}
\newtheorem{ex}[thm]{Example}

\theoremstyle{definition}

\makeatletter
\renewcommand*\env@matrix[1][*\c@MaxMatrixCols c]{%
  \hskip -\arraycolsep
  \let\@ifnextchar\new@ifnextchar
  \array{#1}}
\newcommand*\isom{\xrightarrow{\sim}}
\newcommand{\pair}[1]{\langle#1\rangle}

\newcommand{\divisor}{\operatorname{div}}
\newcommand{\Div}{\operatorname{Div}}
\newcommand{\mult}{\operatorname{mult}}
\newcommand{\ord}{\operatorname{ord}}
\newcommand{\length}{\operatorname{length}}
\newcommand{\Pic}{\operatorname{Pic}}

\newcommand{\CH}{\operatorname{CH}}

\newcommand{\Spec}{\operatorname{Spec}}
\newcommand{\Aut}{\operatorname{Aut}}

\newcommand{\Sym}{\operatorname{Sym}}

\newcommand{\Sp}{\operatorname{Sp}}

\def\qq{\mathbb{Q}}
\def\PP{\mathbb{P}}
\def\rr{\mathbb{R}}
\def\zz{\mathbb{Z}}
\def\cc{\mathbb{C}}

\def\UU{\mathbb{U}}
\def\DD{\mathbb{D}}
\def\ss{\mathcal{S}}
\def\nn{\mathcal{N}}
\def\HH{\mathbb{H}}
\def\mm{\mathcal{M}}
\def\ll{\mathcal{L}}
\def\bb{\mathcal{B}}
\def\QQ{\mathcal{Q}}
\def\CC{\mathcal{C}}

\def\gg{\mathcal{G}}

\def\oo{\mathcal{O}}
\def\aa{\mathcal{A}}

\def\ee{\mathcal{E}}
\def\ff{\mathcal{F}}
\def\uu{\mathcal{U}}
\def\xx{\mathcal{X}}

\def\Im{\mathrm{Im}\,}

\def\d{\mathrm{d}}

\def\eps{\epsilon}

\def\h{\mathrm{h}}

\def\Hdg{\mathrm{Hdg}}

\numberwithin{equation}{section}

\newcounter{nootje}
\begin{document}

\title[On the height of Gross-Schoen cycles in genus three]{On the height of Gross-Schoen cycles in genus three}

\author{Robin de Jong}

\begin{abstract} We show that there exists a sequence of genus three curves defined over the rationals in which the height of a canonical Gross-Schoen cycle tends to infinity.
\end{abstract}

\maketitle

\thispagestyle{empty}

\section{Introduction}
\renewcommand*{\thethm}{\Alph{thm}}

Let $X$ be a smooth, projective and geometrically connected curve of genus $g \geq 2$ over a field $k$ and let $\alpha$ be a divisor of degree one on $X$. The Gross-Schoen cycle $\Delta_\alpha$ associated to $\alpha$ is a modified diagonal cycle in codimension two on the triple product $X^3$, studied in detail in \cite{gs} and \cite{zhgs}. The cycle $\Delta_\alpha$ is homologous to zero, and its class in $\CH^2(X^3)$ depends only on the class of $\alpha$ in $\Pic^1 X$.  

Assume that $k$ is a number field or a function field of a curve. Gross and Schoen show in \cite{gs} the existence of a Beilinson-Bloch height $\pair{\Delta_\alpha,\Delta_\alpha} \in \rr$ of the cycle $\Delta_\alpha$, under the assumption that $X$ has a ``good'' regular model over $k$. A good regular model exists after a suitable finite extension of the base field $k$, and one can unambiguously define a height $\pair{\Delta_\alpha,\Delta_\alpha}$ of the Gross-Schoen cycle for all $X$ over $k$ and all $\alpha \in \Div^1 X$ by passing to a finite extension of $k$ where $X$ has a good regular model, computing the Beilinson-Bloch height over that extension,  and dividing by the degree of the extension. 

Standard arithmetic conjectures of Hodge Index type \cite{giso} predict that one should always have the inequality $\pair{\Delta_{\alpha},\Delta_{\alpha}} \geq 0$, and that equality should hold if and only if the class of the cycle $\Delta_{\alpha}$ vanishes in $\CH^2(X^3)_\qq$. 
Zhang \cite{zhgs} has proved formulae that connect the height $\pair{\Delta_\alpha,\Delta_\alpha}$ of a Gross-Schoen cycle with more traditional invariants of $X$, namely the stable self-intersection of the relative dualizing sheaf, and the stable Faltings height. Zhang's formulae feature some new interesting local invariants of $X$, called the $\varphi$-invariant and the $\lambda$-invariant. 

For $\alpha \in \Div^1 X$ let $x_\alpha$ be the class of the divisor $\alpha - K_X/(2g-2)$ in $\Pic^0(X)_\qq$, where $K_X$ is a canonical divisor on $X$. Then a \emph{canonical} Gross-Schoen cycle on $X^3$ is a Gross-Schoen cycle $\Delta_\alpha$ for which the class $x_\alpha$ vanishes in $\Pic^0(X)_\qq$. 
A corollary of Zhang's formulae in \cite{zhgs} is that for given $X$, the height $\pair{\Delta_{\alpha},\Delta_{\alpha}}$ is minimized for $\Delta_\alpha$ a canonical Gross-Schoen cycle. The question as to the non-negativity of $\pair{\Delta_{\alpha},\Delta_{\alpha}}$ is therefore reduced to the cases where $\Delta_\alpha$  is canonical. 

As an example, for $X$ a hyperelliptic curve and $\alpha$ a Weierstrass point on $X$ one has by \cite[Proposition~4.8]{gs}  that $\Delta_\alpha$ is zero in $\CH^2(X^3)_\qq$. It follows that the height $\pair{\Delta_{\alpha},\Delta_{\alpha}}$ vanishes, and by Zhang's formulae the height of any Gross-Schoen cycle on a hyperelliptic curve is non-negative.

When $k$ is a function field in characteristic zero the inequality $\pair{\Delta_{\alpha},\Delta_{\alpha}} \geq 0$ is known to hold by an application of the Hodge Index Theorem \cite{zh2}. It seems that only very little is known though beyond the hyperelliptic case when $k$ is a function field in positive characteristic, or a number field. K.\ Yamaki shows  in \cite{ya} that $\pair{\Delta_{\alpha},\Delta_{\alpha}} \geq 0$ if $X$ is a non-hyperelliptic curve of genus three with semistable reduction over a function field, under the assumption that certain topological graph types do not occur as dual graph of a special fiber of the semistable regular model. 

The purpose of this paper is to prove the following theorem.
\begin{thm} \label{main_intro} There exists a sequence of genus three curves over $\qq$ in which the height of a canonical Gross-Schoen cycle tends to infinity.
\end{thm}
To the best of the author's knowledge, Theorem \ref{main_intro} is the first result to prove unconditionally the existence of a curve $X$ over a number field such that a canonical Gross-Schoen cycle on $X^3$ has strictly positive height. 

Our proof of Theorem \ref{main_intro} is, like Yamaki's work, based on Zhang's formulae. More precisely we use the formula that relates the height of a canonical Gross-Schoen cycle on $X^3$ with the stable Faltings height of $X$. We then express the Faltings height of a non-hyperelliptic curve of genus three in terms of the well-known modular form $\chi_{18}$ of level one and weight~$18$, defined over $\zz$. Combining both results we arrive at an expression for the height of a canonical Gross-Schoen cycle on a non-hyperelliptic genus three curve $X$ with semistable reduction as a sum of local contributions ranging over all places of $k$, cf. Theorem \ref{formula_GS_genusthree}. 

The local non-archimedean contributions can be bounded from below by some combinatorial data in terms of the dual graphs associated to the stable model of $X$ over $k$. This part of the argument is heavily inspired by Yamaki's work \cite{yaf} dealing with the function field case. In fact, the differences with \cite{yaf} at this point are only rather small: the part of \cite{yaf} that works only in a global setting, by an application of the Hirzebruch-Riemann-Roch theorem, is replaced here by a more local approach, where the application of Hirzebruch-Riemann-Roch is replaced by an application of Mumford's functorial Riemann-Roch \cite{mu}. 

The modular form $\chi_{18}$ is not mentioned explicitly in \cite{yaf} but clearly plays a role in the background. As an intermediate result, we obtain an expression for the local order of vanishing of $\chi_{18}$ in terms of the Horikawa index \cite{ko} \cite{reid} and the discriminant, cf.\ Proposition \ref{chi_and_Ind}. This result might be of independent interest.

We will then pass to a specific family of non-hyperelliptic genus three curves $C_n$ defined over $\qq$, for $n \in \zz_{>0}$ and $n \to \infty$, considered by J.\ Gu\`ardia in  \cite{gu}. In the paper \cite{gu}, the stable reduction types of the curves $C_n$ are determined explicitly. By going through the various cases, we will see that the local non-archimedean contributions to the height of a canonical Gross-Schoen cycle on $C_n$, as identified by Theorem \ref{formula_GS_genusthree}, are all non-negative. 

To deal with the archimedean contribution, we observe that the curves $C_n$ are all fibers of the family of smooth curves
\[ D_\kappa \colon \, y^4 = x(x-1)(x-\kappa) \, , \, \kappa \in \PP^1  \setminus  \{0,1,\infty\} \, .  \]
 The family $D_\kappa$  is rather special and has been studied in detail by various authors, see for instance G.\ Forni \cite{fo}, F.\ Herrlich and G.\ Schmith\"usen \cite{hs}, and M.\ M\"oller \cite{mo}. As is shown in these references, the family $D_\kappa$ gives rise to a Teichm\"uller curve in $\mm_3$, and to a Shimura curve in $\aa_3$. Let $E_\kappa$ denote the elliptic curve $y^2 = x(x-1)(x-\kappa)$. Then for each $\kappa \in \PP^1 \setminus  \{0,1,\infty\}$, the jacobian of $D_\kappa$  is isogenous to the product $E_\kappa \times E_{-1} \times E_{-1}$, by \cite[Proposition~2.3]{gu} or \cite[Proposition~7]{hs}. 

We show that the archimedean contribution to the height of a canonical Gross-Schoen cycle of $C_n$ is bounded from below by a quantity that tends to infinity like $\log n$. In order to do this we recall the work \cite{hrar} by R.\ Hain and D.\ Reed on the Ceresa cycle, which allows us to study the archimedean contribution as a  function of $\kappa \in \PP^1 \setminus \{0,1,\infty\}$. For $n \to \infty$ we have $\kappa \to 0$. The stable reduction of the family $D_\kappa$ near $\kappa=0$ is known, see for instance \cite[Proposition~8]{hs}, and the asymptotic behavior of the archimedean contribution near $\kappa=0$ can then be determined by invoking an asymptotic result due to P.~Brosnan and G.~Pearlstein \cite{bp}. 

The paper is organized as follows. In Sections \ref{sec:nonarch} and \ref{sec:arch} we recall the non-archimedean and archimedean $\varphi$- and $\lambda$-invariants from Zhang's paper \cite{zhgs}. The main formulae from \cite{zhgs} relating the height of the Gross-Schoen cycle to the self-intersection of the relative dualizing sheaf and the Faltings height are then stated in Section \ref{sec:zhang_formulae}. In Section \ref{sec:some_lambda} we display Zhang's $\lambda$-invariant for a couple of polarized metrized graphs that we will encounter in our proof of Theorem \ref{main_intro}.

In Section \ref{sec:modular_form} we recall a few general results on analytic and algebraic modular forms, and in Section \ref{sec:biext} we recall the work of Hain and Reed, and Brosnan and Pearlstein that we shall need on the asymptotics of the archimedean contribution to the height. In Section \ref{chi} we discuss the modular form $\chi_{18}$. The first new results are contained in Section \ref{sec:horikawa}, where we recall the Horikawa index for stable curves in genus three and show how it can be expressed in terms of the order of vanishing of $\chi_{18}$ and the discriminant. This leads to a useful lower bound for the order of vanishing of $\chi_{18}$. Sections \ref{sec:proof_main}--\ref{sec:archimedean} contain the proof of Theorem \ref{main_intro}. 

\renewcommand*{\thethm}{\arabic{section}.\arabic{thm}}

\section{Non-archimedean invariants} \label{sec:nonarch}

We introduce metrized graphs and their polarizations, and explain how a stable curve over a discrete valuation ring canonically gives rise to a polarized metrized graph (pm-graph). References for this section are for example \cite[Sections~3 and~4]{ci_eff} \cite[Section~1]{ya} \cite[Appendix]{zhadm} \cite[Section~4]{zhgs}. In this paper, a \emph{metrized graph}  is a connected compact metric space $\Gamma$ such that $\Gamma$ is either a point or for each $p \in \Gamma$ there exist a positive integer $n$ and $\eps \in \rr_{>0}$ such that $p$ possesses an open neighborhood $U$ together with an isometry $U \isom S(n,\eps)$, where $S(n,\eps)$ is the star-shaped set 
\[ S(n,\eps)= \{ z \in \cc \, : \,   \textrm{there exist} \,\, 0 \leq t < \eps \,\, \textrm{and}\, \, k \in \zz \,\, \textrm{such that} \,\, z = t e^{2 \pi i k/n} \} \, , \]
endowed with the path metric. If $\Gamma$ is a metrized graph, not a point, then for each $p \in \Gamma$ the integer $n$ is uniquely determined, and is called the \emph{valence} of $p$, notation $v(p)$. We set the valence of the unique point of the point-graph to be zero. Let $V_0 \subset \Gamma$ be the set of points $p \in \Gamma$ with $v(p) \neq 2$. Then $V_0$ is a finite subset of $\Gamma$, and we call any finite non-empty set $V \subset \Gamma$ containing $V_0$ a \emph{vertex set} of $\Gamma$.   

Let $\Gamma$ be a metrized graph and let $V$ be a vertex set of $\Gamma$. Then $\Gamma \setminus V$ has a finite number of connected components, each isometric with an open interval. The closure in $\Gamma$ of a connected component of $\Gamma \setminus V$ is called an \emph{edge} associated to $V$. We denote by $E$ the set of edges of $\Gamma$ resulting from the choice of $V$. When $e \in E$ is obtained by taking the closure in $\Gamma$ of the connected component $e^\circ$ of $\Gamma \setminus V$ we call $e^\circ$ the \emph{interior} of $e$. The assignment $e \mapsto e^\circ$ is unambiguous, given the choice of $V$, as we have $e^\circ = e \setminus V$. We call $e \setminus e^\circ$ the set of \emph{endpoints} of $e$. For example, assume $\Gamma$ is a circle, and say $V$ consists of $n>0$ points on $\Gamma$. Then $\Gamma \setminus V$ has $n$ connected components, and $\Gamma$ has $n$ edges. In general, an edge is homeomorphic to either a circle or a closed interval, and thus has either one endpoint or two endpoints. 

Let $e \in E$, and assume that $e^\circ$ is isometric with the open interval $(0,\ell(e))$. Then the positive real number $\ell(e)$ is well-defined and called the \emph{weight} of $e$. The total weight $\delta(\Gamma)=\sum_{e \in E} \ell(e)$ is called the \emph{volume} of $\Gamma$. We note that the volume $\delta(\Gamma)$ of a metrized graph $\Gamma$ is independent of the choice of a vertex set $V$. 

A divisor on $\Gamma$ is to be an element of $\zz^V$. A divisor on $\Gamma$ has a natural \emph{degree} in $\zz$. Assume we have fixed a map $\mathbf{q} \colon V \to \zz$. The associated \emph{canonical divisor} $K=K_\mathbf{q}$ is by definition the element $K \in \zz^V$ such that for all $p \in V$ the equality $K(p)=v(p)-2+2 \,\mathbf{q}(p)$ holds. We call the pair $\overline{\Gamma}=(\Gamma,\mathbf{q})$ a \emph{polarized metrized graph}, abbreviated \emph{pm-graph}, if $\mathbf{q}$ is non-negative, and the canonical divisor $K_\mathbf{q}$ is effective. Let $\overline{\Gamma}=(\Gamma, \mathbf{q})$ be a pm-graph with vertex set $V$. We call the integer
\[ g = g(\overline{\Gamma}) = \frac{1}{2}(\deg K +2 ) = b_1(\Gamma) + \sum_{p \in V} \mathbf{q}(p) \]
the \emph{genus} of $\overline{\Gamma}$. Here $b_1(\Gamma) \in \zz_{\geq0}$ is the first Betti number of $\Gamma$. We see that $g(\overline{\Gamma}) \in \zz_{\geq 1}$. We occasionally call $\mathbf{q}(p)$ the \emph{genus} of the vertex $p \in V$.

An edge $e \in E$ is called \emph{of type $0$} if removal of its interior results into a connected graph. Let $h \in [1,g/2]$ be an integer.  An edge $e \in E$ is called \emph{of type $h$} if removal of its interior yields the disjoint union of a pm-graph of genus $h$ and a pm-graph of genus $g-h$. The total weight of edges of type $0$ is denoted by $\delta_0(\overline{\Gamma})$, and the total weight of edges of type $h$ is denoted $\delta_h(\overline{\Gamma})$. We have $\delta(\Gamma) = \sum_{h=0}^{[g/2]} \delta_h(\overline{\Gamma})$. 

We refer to \cite{zhadm} for the definition of the admissible measure $\mu$ on $\Gamma$ associated to the divisor $K=K_\mathbf{q}$, and the admissible Green's function $g_\mu \colon \Gamma \times \Gamma \to \rr$. 
We will be interested in the following invariants, all introduced by Zhang \cite{zhadm} \cite{zhgs}. First of all, we consider the $\varphi$-invariant,
\begin{equation} \label{phi_nonarch} \varphi(\overline{\Gamma}) = -\frac{1}{4}\delta(\Gamma) + \frac{1}{4} \int_\Gamma g_\mu(x,x) ((10g+2) \mu(x) - \delta_K(x)) \, .  
\end{equation}
Next we consider the $\epsilon$-invariant,
\begin{equation} \epsilon(\overline{\Gamma}) = \int_\Gamma g_\mu(x,x)((2g-2) \mu(x) + \delta_K(x)) \, .  
\end{equation} 
Finally we consider the $\lambda$-invariant, 
\begin{equation} \label{lambda_nonarch}
\lambda(\overline{\Gamma}) =  \frac{g-1}{6(2g+1)} \varphi(\overline{\Gamma}) + \frac{1}{12}\left( \delta(\Gamma) + \epsilon(\overline{\Gamma})\right)\, . 
\end{equation}

Let $S=\{e_1,\ldots,e_n\}$ be a subset of $E$. We define $\Gamma_{ \{e_1 \} }$ to be the topological space obtained from $\Gamma$ by contracting the subspace $e_1$ to a point. Then $\Gamma_{ \{e_1 \} }$ has a natural structure of metrized graph, and the natural projection $\Gamma \to \Gamma_{ \{e_1\} }$ endows $\Gamma_{ \{e_1\} }$ with a designated vertex set, and maps each edge $e_i$ for $i=2,\ldots,n$ onto an edge of $\Gamma_{ \{e_1\} }$. Continuing by induction we obtain after $n$ steps a metrized graph $\Gamma_S$ with  natural projection $\pi \colon \Gamma \to \Gamma_S$ and designated vertex set $V_S$. The result is independent of the ordering of the edges in $S$ and is called the metrized graph obtained by \emph{contracting} the edges in $S$.

Consider the pushforward divisor $\pi_* K_\mathbf{q}$ on $\Gamma_S$. It is then clear that $\pi_* K_\mathbf{q}$ is effective and has the same degree as $K_{\mathbf{q}}$. The associated map $\mathbf{q}_S \colon V_S \to \zz$ is non-negative, and thus we obtain a pm-graph $\overline{\Gamma}_S=(\Gamma_S,\mathbf{q}_S)$ canonically determined by $S$. Clearly we have $g(\overline{\Gamma}_S)=g(\overline{\Gamma})$.
The pm-graph obtained by contracting all edges in $E \setminus S$ is denoted by $\overline{\Gamma}^S=(\Gamma^S,\mathbf{q}^S)$. 

Assume $\Gamma$ is not a point. When $\Gamma_1, \Gamma_2$ are subgraphs of $\Gamma$ such that $\Gamma=\Gamma_1 \cup \Gamma_2$ and $\Gamma_1 \cap \Gamma_2$ consists of one point, we say that $\Gamma$ is the \emph{wedge sum} of $\Gamma_1, \Gamma_2$, notation $\Gamma = \Gamma_1 \vee \Gamma_2$. By induction one has a well-defined notion of wedge sum $\Gamma_1 \vee \ldots \vee \Gamma_n$ of subgraphs $\Gamma_1,\ldots,\Gamma_n$ of $\Gamma$. We say that $\Gamma$ is \emph{irreducible} if the following holds: write $\Gamma = \Gamma_1 \vee \Gamma_2$ as a wedge sum. Then one of $\Gamma_1, \Gamma_2$ is a one-point graph. The graph $\Gamma$ has a unique decomposition $\Gamma=\Gamma_1 \vee \ldots \vee \Gamma_n$ as a wedge sum of irreducible subgraphs. We call the $\Gamma_i$ the \emph{irreducible components} of $\Gamma$. Each $\Gamma_i$ can be canonically seen as the contraction of some edges of $\Gamma$, and hence has a natural induced structure of pm-graph $\overline{\Gamma}_i$ of genus $g$, where $g=g(\overline{\Gamma})$ is the genus of $\overline{\Gamma}$. 

We call an invariant $\kappa=\kappa(\overline{\Gamma})$ of pm-graphs of genus $g$ \emph{additive} if the invariant $\kappa$ is compatible with decomposition into irreducible components. More precisely, let $\overline{\Gamma}$ be a pm-graph of genus $g$ and let $\overline{\Gamma} = \overline{\Gamma}_1 \vee \ldots \vee \overline{\Gamma}_n$ be its decomposition into irreducible components, where each $\overline{\Gamma}_i$ has its canonical induced structure of pm-graph of genus $g$. Then we should have $\kappa(\overline{\Gamma})= \kappa(\overline{\Gamma}_1) + \cdots + \kappa(\overline{\Gamma}_n)$. It is readily seen that each of the invariants $\delta_h(\overline{\Gamma})$ where $h=0,\ldots,[g/2]$ is additive on pm-graphs of genus $g$. By \cite[Theorem~4.3.2]{zhgs} the $\varphi$-invariant, the $\epsilon$-invariant  and the $\lambda$-invariant are all additive on pm-graphs of genus $g$.

Let $G=(V,E)$ be a connected graph (multiple edges and loops are allowed) and let $\ell \colon E \to \rr_{>0}$ be a function on the edge set $E$ of $G$. We then call the pair $(G,\ell)$ a \emph{weighted graph}. Let $(G,\ell)$ be a weighted graph. Then to $(G,\ell)$ one has naturally associated a metrized graph $\Gamma$ by glueing together finitely many closed intervals $I(e)=[0,\ell(e)]$, where $e$ runs through $E$, according to the vertex assignment map of $G$. Note that the resulting metrized graph $\Gamma$ comes equipped with a distinguished vertex set $V \subset \Gamma$.

Let $R$ be a discrete valuation ring and write $S=\Spec R$. Let $f \colon \xx \to S$ be a generically smooth stable curve of genus $g \geq 2$ over $S$. We can canonically attach a weighted graph $(G,\ell)$ to $f$ in the following manner. Let $C$ denote the geometric special fiber of $f$. Then the graph $G$ is to  be the dual graph of $C$. Thus the vertex set $V$ of $G$ is the set of irreducible components of $C$, and the edge set $E$ is the set of nodes of $C$. The incidence relation of $G$ is determined by sending a node $e$ of $C$ to the set of irreducible components of $C$ that $e$ lies on. Each $e \in E$ determines a closed point on $\xx$. We let $\ell(e) \in \zz_{>0}$ be its so-called \emph{thickness} on $\xx$. 

Let $\Gamma$ denote the metrized graph associated to $(G,\ell)$ with designated vertex set~$V$. We have a canonical map $\mathbf{q} \colon V \to \zz$ given by associating to $v \in V$ the geometric genus of the irreducible component $v$. The map $\mathbf{q}$ is non-negative, and the associated canonical divisor $K_\mathbf{q}$ is effective. We therefore obtain a canonical pm-graph $\overline{\Gamma}=(\Gamma,\mathbf{q})$ from $f$. The genus $g(\overline{\Gamma})$ is equal to the genus of the generic fiber of $f$. 

Let $J \colon S \to \overline{\mm}_g$ denote the classifying map to the moduli stack of stable curves of genus $g$ determined by $f$. For $h=0,\ldots,[g/2]$ we have canonical boundary divisors $\Delta_h$ on $\overline{\mm}_g$ whose generic points correspond to irreducible stable curves of genus~$g$ with one node (in the case $h=0$), or to reducible stable curves consisting of two irreducible components of genus $h$ and $g-h$, joined at one point (in the case $h>0$). Let $v$ denote the closed point of $S$. Then for each $h=0,\ldots, [g/2]$ we have the equality
\begin{equation} \delta_h(\overline{\Gamma}) = \mult_v \, J^* \Delta_h 
\end{equation}
in $\zz$, connecting the combinatorial structure of $\overline{\Gamma}$ with the geometry of $\overline{\mm}_g$.

\section{Archimedean invariants} \label{sec:arch}

In \cite{zhgs} Zhang introduces archimedean analogues of the $\varphi$-invariant and $\lambda$-invariant from (\ref{phi_nonarch}) and (\ref{lambda_nonarch}). Let $C$ be a compact and connected Riemann surface of genus $g \geq 2$. Let $\mathrm{H}^0(C,\omega_C)$ denote the space of holomorphic differentials on $C$, equipped with the hermitian inner product
\begin{equation} \label{defhodgemetric} (\alpha,\beta) \mapsto \frac{i}{2} \int_C \alpha \wedge \overline{\beta} \, . 
\end{equation}
We denote the resulting norm on $\det \mathrm{H}^0(C,\omega_C)$ by $\|\cdot\|_{\Hdg}$.
Choose an orthonormal basis $(\eta_1,\ldots,\eta_g)$ of $\mathrm{H}^0(C,\omega_C)$, and put
\[ \mu_C = \frac{i}{2g} \sum_{k=1}^g \eta_k \wedge \overline{\eta}_k \, , \]
following Arakelov in \cite{ar}. Then $\mu_C$ is a volume form on $C$. Let $\Delta_\mathrm{Ar}$ be the Laplacian operator on $L^2(C,\mu_C)$, i.e.\ the endomorphism of  $L^2(C,\mu_C)$ determined by setting
\[ \frac{ \partial \overline{\partial}}{\pi i} f = \Delta_\mathrm{Ar}(f) \cdot \mu_C  \]
for $f \in L^2(C,\mu_C)$. The differential operator $\Delta_\mathrm{Ar}$ is positive elliptic and hence has a discrete spectrum $0=\lambda_0<\lambda_1 \leq \lambda_2 \leq \ldots$ of real eigenvalues, where each eigenvalue occurs with finite multiplicity. Moreover, one has an orthonormal basis $(\phi_k)_{k=0}^\infty$ of $L^2(C,\mu_C)$ where $\phi_k$ is an eigenfunction of $\Delta_\mathrm{Ar}$ with eigenvalue $\lambda_k$ for each $k=0,1,2,\ldots$.  The $\varphi$-invariant $\varphi(C)$ of $C$ is then defined to be the real number
\begin{equation} \varphi(C) = \sum_{k>0} \frac{2}{\lambda_k} \sum_{m,n=1}^g \left|\int_C \phi_k \cdot \eta_m \wedge \overline{\eta}_n \right|^2 \, . 
\end{equation}
We note that this invariant was also introduced and studied independently by N.~Kawazumi in \cite{kaw}. One has $\varphi(C)>0$, see \cite[Corollary 1.2]{kaw} or \cite[Remark following Proposition 2.5.3]{zhgs}.

Let $\delta_F(C)$ be the delta-invariant of $C$ as defined by G.\ Faltings in \cite[p.~401]{fa}, and put $\delta(C)=\delta_F(C)-4g \log(2\pi)$. Then the $\lambda$-invariant $\lambda(C)$ of $C$ is defined to be the real number 
 \begin{equation} \label{lambda_arch}  \lambda(C) = \frac{g-1}{6(2g+1)} \varphi(C) + \frac{1}{12}\delta(C) \, . 
 \end{equation}
Note the similarity with (\ref{lambda_nonarch}). For fixed $g \geq 2$, both $\varphi$ and $\lambda$ are $C^\infty$ functions on the moduli space of curves $\mm_g(\cc)$. Some of their properties (for instance Levi form and asymptotic behavior near generic points of the boundary) are found in the references \cite{djtorus} \cite{djkzasympt} \cite{djnormal} \cite{djsecond} \cite{kaw}.

\section{Zhang's formulae for the height of the Gross-Schoen cycle} \label{sec:zhang_formulae}

The non-archimedean and archimedean $\varphi$- and $\lambda$-invariants as introduced in the previous two sections occur in \cite{zhgs} in formulae relating the height of a Gross-Schoen cycle on a curve over a global field with more traditional invariants, namely the self-intersection of the relative dualizing sheaf, and the Faltings height, respectively. The purpose of this section is to recall these formulae. In view of our applications, we will be solely concerned here with the number field case.

Let $k$ be a number field and let $X$ be a smooth projective geometrically connected curve of genus $g \geq 2$ defined over $k$. Let $\alpha \in \Div^1 X$ be a divisor of degree one on $X$. Following \cite[Section~1.1]{zhgs} we have an associated Gross-Schoen cycle $\Delta_{\alpha}$ in the rational Chow group $\CH^2(X^3)_\qq$. The cycle $\Delta_\alpha$ is homologous to zero, and has by \cite{gs} a well-defined Beilinson-Bloch height $\pair{\Delta_\alpha,\Delta_\alpha} \in \rr$. The height $\pair{\Delta_\alpha,\Delta_\alpha} $ vanishes if $\Delta_\alpha$ is rationally equivalent to zero.

Assume now that $X$ has semistable reduction over $k$. Let $\hat{\omega}$ denote the admissible  relative dualizing sheaf of $X$ from \cite{zhadm}, viewed as an adelic line bundle on $X$. Let $ \pair{\hat{\omega},\hat{\omega} } \in \rr$ be its self-intersection as in \cite{zhadm}.  
Let $O_k$ be the ring of integers of $k$. Denote by $M(k)_0$ the set of finite places of $k$, and by $M(k)_\infty$ the set of complex embeddings of $k$. We set $M(k) = M(k)_0 \sqcup M(k)_\infty$. For $v \in M(k)_0$ we set $Nv$ to be the norm of the residue field of $O_k$ at $v$, and for $v \in M(k)_\infty$ we set $Nv=1$. 

Let $S=\Spec O_k$ and let $f \colon \xx \to S$ denote the stable model of $X$ over $S$. For $v \in M(k)_0$ we denote by $\varphi(X_v)$ the $\varphi$-invariant of the pm-graph of genus $g$ canonically associated to the base change of $f \colon \xx \to S$ along the inclusion $O_k \to O_{k,v}$. For $v \in M(k)_\infty$ we denote by $\varphi(X_v)$ the $\varphi$-invariant of the compact and connected Riemann surface $X_v = X \otimes_v \cc$ of genus $g$.

Let $x_\alpha$ be the class of the divisor $\alpha - K_X/(2g-2)$ in $\Pic^0(X)_\qq$, where $K_X$ is a canonical divisor on $X$. Let $\hat{\h}$ denote the canonical N\'eron-Tate height on $\Pic^0(X)_\qq$. With these notations Zhang has proved the following identity \cite[Theorem~1.3.1]{zhgs}.
\begin{thm} \label{zhang_main} (S. Zhang \cite{zhgs}) Let $X$ be a smooth projective geometrically connected curve of genus $g \geq 2$ defined over the number field $k$. Let $\alpha \in \Div^1 X$ be a divisor of degree one on $X$, and assume that $X$ has semistable reduction over $k$.  Then the equality
\[ \pair{\Delta_{\alpha},\Delta_{\alpha}} = \frac{2g+1}{2g-2} \pair{\hat{\omega},\hat{\omega} }- \sum_{v \in M(k)} \varphi(X_v) \log Nv  + 12 (g-1)\, [k\colon \qq] \, \hat{\h}(x_\alpha)  \]
holds in $\rr$.
\end{thm}
 We see from Theorem \ref{zhang_main} that for fixed $X$, the height $\pair{\Delta_\alpha,\Delta_\alpha}$ attains its minimum precisely when $x_\alpha$ is zero in $\Pic^0(X)_\qq$. We refer to $\Delta_\alpha$ where $x_\alpha$ is zero as a \emph{canonical} Gross-Schoen cycle. Also, by Theorem \ref{zhang_main}, the non-negativity of the height of a canonical Gross-Schoen cycle (as predicted by standard arithmetic conjectures of Hodge Index type \cite{giso})  is equivalent to the lower bound
\begin{equation} \label{conj}  (?) \qquad \pair{\hat{\omega},\hat{\omega} } \geq \frac{2g-2}{2g+1} \sum_{v \in M(k)} \varphi(X_v) \log Nv  
\end{equation}
for the self-intersection of the admissible relative dualizing sheaf. 

We recall that the strict inequality $\pair{\hat{\omega},\hat{\omega} } > 0$ is equivalent to the Bogomolov conjecture for $X$, canonically embedded in its jacobian. A conjecture by Zhang \cite[Conjecture~4.1.1]{zhgs}, proved by Z.~Cinkir \cite[Theorem~2.9]{ci_eff}, implies that for $v \in M(k)_0$ one has $\varphi(X_v) \geq 0$. As $\varphi(X_v)>0$ for $v \in M(k)_\infty$ we find that the right hand side of (\ref{conj}) is strictly positive. Hence, the non-negativity of the height of a canonical Gross-Schoen cycle implies the Bogomolov conjecture for $X$.

We mention that in \cite[Corollary~1.4]{djnt} it is shown unconditionally that the inequality
\[   \pair{\hat{\omega},\hat{\omega} } \geq \frac{2}{3g-1}  \sum_{v \in M(k)} \varphi(X_v) \log Nv  \]
holds. This inequality is weaker than (\ref{conj}) if $g \geq 3$ but still implies the Bogomolov conjecture for $X$.

We next discuss the connection with the Faltings height.
Let $(\ll, (\|\cdot\|_v)_{v \in M(k)_\infty})$ be a metrized line bundle on $S$. Its arithmetic degree is given by choosing a non-zero rational section $s$ of $\ll$ and by setting
\begin{equation} \label{arithm_degree} 
\deg \left( \ll, (\|\cdot\|_v)_{v \in M(k)_\infty} \right)  = \sum_{v \in M(k)_0} \ord_v (s) \log Nv - \sum_{v \in M(k)_\infty} \log \|s\|_v \, . 
\end{equation}
The arithmetic degree is independent of the choice of section $s$, by the product formula. As before let $f \colon \xx \to S$ denote the stable model of $X$ over $S$. Let $\omega_{\xx/S}$ denote the relative dualizing sheaf on $\xx$. We endow the line bundle $\det f_*\omega_{\xx/S}$ on $S$ with the metrics $\|\cdot\|_{\Hdg, v}$ at the infinite places determined by the inner product in (\ref{defhodgemetric}). The resulting metrized line bundle is denoted $\det f_*\bar{\omega}_{\xx/S}$. Its arithmetic degree $\deg \det f_*\bar{\omega}_{\xx/S}$ is the (non-normalized) stable Faltings height of $X$.

Let $\pair{\bar{\omega},\bar{\omega} } $ denote the Arakelov self-intersection of the relative dualizing sheaf on $\xx$. The Noether formula \cite[Theorem~6]{fa} \cite[Th\'eor\`eme~2.5]{mb} then states that
\begin{equation} \label{noether}  12  \deg \det f_*\bar{\omega}_{\xx/S} = \pair{\bar{\omega},\bar{\omega} } + \sum_{v \in M(k)} \delta(X_v) \log Nv  \, . 
\end{equation}
Here, for $v \in M(k)_0$ we denote by $\delta(X_v)$ the volume of the pm-graph associated to the base change of $f \colon \xx \to S$ along the inclusion $O_k \to O_{k,v}$, and for $v \in M(k)_\infty$ we denote by $\delta(X_v)=\delta_F(X_v)-4g \log(2\pi)$ the (renormalized) delta-invariant of the compact and connected Riemann surface $X_v = X \otimes_v \cc$ of genus $g$.

Similarly to $\varphi(X_v)$ and $\delta(X_v)$ one also defines $\epsilon(X_v)$ (for $v \in M(k)_0$) and $\lambda(X_v)$. The Arakelov self-intersection of the relative dualizing sheaf on $\xx$ and the self-intersection of the admissible relative dualizing sheaf of $X$ are related by the identity
\begin{equation} \label{omegasq}  \pair{\bar{\omega},\bar{\omega} }  = \pair{\hat{\omega},\hat{\omega} } + \sum_{v \in M(k)_0} \epsilon(X_v)\log Nv \, , 
\end{equation}
cf. \cite[Theorem~4.4]{zhadm}. 
Combining Theorem \ref{zhang_main} with (\ref{lambda_nonarch}), (\ref{lambda_arch}), (\ref{noether}) and (\ref{omegasq}) we find the following alternative formula for the height of a canonical Gross-Schoen cycle  (cf. \cite[Equation~1.4.2]{zhgs}).
\begin{cor} \label{zhang} Let $X$ be a smooth projective geometrically connected curve of genus $g \geq 2$ defined over the number field $k$. Assume that $X$ has semistable reduction over $k$. Let $\Delta \in \CH^2(X^3)_\qq$ be a canonical Gross-Schoen cycle on $X^3$. Then the equality
\[ \pair{\Delta,\Delta} = \frac{6(2g+1)}{g-1} \left(  \deg \det f_* \bar{\omega}_{\xx/S} - \sum_{v \in M(k)} \lambda(X_v) \log Nv \right)  \]
holds in $\rr$.
\end{cor} 

\section{The $\lambda$-invariants for some pm-graphs} \label{sec:some_lambda}

The purpose of this section is to display the $\lambda$-invariants of a few pm-graphs that we will encounter in the sequel. We refer to the papers \cite{ci_adm} \cite{ci_comp}  \cite{ci_eff} by Z.\ Cinkir for an extensive study of the $\varphi$- and $\lambda$-invariants of pm-graphs. The reference \cite{ci_adm} focuses in particular on pm-graphs of genus three. 

Let $\Gamma$ be a metrized graph. Let $r(p,q)$ denote the effective resistance between points $p, q \in \Gamma$. Fix a point $p \in \Gamma$.  We then put
\begin{equation} \label{tau} \tau(\Gamma) =\frac{1}{4} \int_{\Gamma}\left(\frac{\partial}{\partial x} r(p,x)\right)^2 \d x \, , 
\end{equation}
where $\d x$ denotes the (piecewise) Lebesgue measure on $\Gamma$. 

By \cite[Lemma~2.16]{cr} the number $\tau(\Gamma)$ is independent of the choice of $p \in \Gamma$. It is readily verified that for a circle $\Gamma$ of length $\delta(\Gamma)$ we have $\tau(\Gamma)=\frac{1}{12}\delta(\Gamma)$, and for a line segment of length $\delta(\Gamma)$ we have $\tau(\Gamma)=\frac{1}{4}\delta(\Gamma)$. The $\tau$-invariant is an additive invariant. 

Now let $\overline{\Gamma}=(\Gamma,\mathbf{q})$ be a pm-graph of genus $g$, with vertex set $V$, and canonical divisor $K$. We set
\[ \theta(\overline{\Gamma}) = \sum_{p, q \in V} (v(p)-2+2\,\mathbf{q}(p))(v(q)-2+2\,\mathbf{q}(q))\, r(p,q) \, .  \]
The next proposition, due to Cinkir, expresses $\lambda(\overline{\Gamma})$ in terms of $\tau(\Gamma)$, $\theta(\overline{\Gamma})$ and the volume $\delta(\Gamma)$. 
\begin{prop}  \label{cinkir} Let $\overline{\Gamma}$ be a pm-graph of genus~$g$. Then the equality
\[  (8g+4) \lambda(\overline{\Gamma}) = 6(g-1)\tau(\Gamma) + \frac{ \theta(\overline{\Gamma})}{2} + \frac{g+1}{2} \delta(\Gamma)  \]
holds in $\rr$.
\end{prop}
\begin{proof} See \cite[Corollary~4.4]{ci_eff}.
\end{proof}
We will need the following particular cases.
\begin{ex} \label{loop} Let $\overline{\Gamma}$ be a pm-graph of genus $g$ consisting of one vertex of genus $g-1$ and with one loop attached of length $\delta(\Gamma)$. Then $\tau(\Gamma)=\frac{1}{12}\delta(\Gamma)$, $\theta(\overline{\Gamma})=0$ and hence $(8g+4)\lambda(\overline{\Gamma})=g\,\delta(\Gamma)$.
\end{ex}
\begin{ex} \label{segment} Let $\overline{\Gamma}$ be a pm-graph of genus $g$ consisting of two vertices of genera $h$ and $g-h$ joined by one edge of length $\delta(\Gamma)$. Then $\tau(\Gamma)=\frac{1}{4}\delta(\Gamma)$, $\theta(\overline{\Gamma})=2(2h-1)(2g-2h-1)\delta(\Gamma)$ and hence $(8g+4)\lambda(\overline{\Gamma})=4h(g-h)\delta(\Gamma)$.
\end{ex}
\begin{ex} \label{tree} Let $\overline{\Gamma}$ be a polarized metrized tree of genus $g$. Then we have
\[  (8g+4) \lambda(\overline{\Gamma}) = \sum_{h=1}^{[g/2]} 4h(g-h) \delta_h (\overline{\Gamma}) \, .  \]
This follows from the additivity of the $\lambda$-invariant  and Example \ref{segment}.
\end{ex}
\begin{ex}  \label{lambda_for_two_gon} Let $\overline{\Gamma}$ be a pm-graph of genus $g$ consisting of two vertices of genera $h$ and $g-h-1$ and joined by two edges of weights $m_1, m_2$.  We have
\[ \tau(\Gamma) = \frac{1}{12}\delta(\Gamma) = \frac{1}{12}(m_1+m_2) \, , \quad \theta(\overline{\Gamma}) =  \frac{8m_1m_2}{m_1+m_2} (g-h-1)h \, , \]
and hence
\[ (8g+4)  \lambda(\overline{\Gamma}) = \frac{4m_1m_2}{m_1+m_2} (g-h-1)h + g(m_1+m_2) \, . \]
\end{ex}

\section{Algebraic and analytic modular forms} \label{sec:modular_form}

References for this section are \cite{dm}, \cite[Chapter~V]{fc}, \cite{geer} and \cite{lrz}. Let $g \geq 1$ be an integer. Let $\aa_g$ be the moduli stack of principally polarized abelian varieties of dimension $g$, and denote by $p \colon \uu_g \to \aa_g$ the universal abelian variety. Let $\Omega_{\uu_g/\aa_g}$ denote the sheaf of relative $1$-forms of $p$. Then we have the Hodge bundle $\ee=p_* \Omega_{\uu_g/\aa_g}$ and its determinant $\ll = \det p_* \Omega_{\uu_g/\aa_g}$ on $\aa_g$. Kodaira-Spencer deformation theory gives a canonical isomorphism 
\begin{equation} \label{KS-ab-schemes} \Sym^2 \ee \isom \Omega_{\aa_g/\zz} 
\end{equation}
of locally free sheaves on $\aa_g$, see e.g.\ \cite[Section~III.9]{fc}.

For all commutative rings $R$ and all $h \in \zz_{\geq 0}$ we let 
\[ \ss_{g,h}(R) = \Gamma(\aa_g \otimes R, \ll^{\otimes h})  \] 
denote the $R$-module of algebraic Siegel modular forms of degree~$g$ and weight~$h$. 

Let $\mathbb{H}_g$ denote Siegel's upper half space of degree~$g$. We have a natural uniformization map $u \colon \mathbb{H}_g \to \aa_g(\cc)$ and hence a universal abelian variety $\tilde{p} \colon \UU_g \to \mathbb{H}_g$ over $\mathbb{H}_g$. The Hodge bundle $\tilde{\ee}=\tilde{p}_*\Omega_{\UU_g/\mathbb{H}_g}$ over $\mathbb{H}_g$ has a standard trivialization by the frame $(\d \zeta_1/\zeta_1,\ldots, \d \zeta_g /\zeta_g) = (2\pi i \,\d z_1,\ldots, 2\pi i \,\d z_g)$, where $\zeta_i = \exp(2\pi i z_i)$. In particular, the determinant of the Hodge bundle  $\tilde{\ll} = \det \tilde{\ee}$ is trivialized by the frame $\omega = \frac{\d \zeta_1}{\zeta_1} \wedge \ldots \wedge \frac{\d \zeta_g}{\zeta_g} = (2\pi i)^g (\d z_1 \wedge \ldots \wedge \d z_g)$.  Let $\mathcal{R}_{g,h}$ denote the usual $\cc$-vector space of analytic Siegel modular forms of degree $g$ and weight $h$. 
Then the map 
\[ \ss_{g,h}(\cc) \to \mathcal{R}_{g,h} \, , \quad s \mapsto \tilde{s} = s \cdot \omega^{\otimes -h} = (2\pi i)^{-gh} \cdot s \cdot (\d z_1 \wedge \ldots \wedge \d z_g)^{\otimes - h} \]
is a linear isomorphism. 

The Hodge metric $\|\cdot\|_\Hdg$ on the Hodge bundle $\tilde{\ee}$ is the metric induced by the standard symplectic form on the natural variation of Hodge structures underlying the local system $R^1\tilde{p}_* \zz_{\UU_g}$ on $\mathbb{H}_g$. The natural induced metric on $\tilde{\ll}$  is given by 
\begin{equation} \label{norm_frame}  \|   \d z_1 \wedge \ldots \wedge \d z_g \|_\Hdg(\Omega) = \sqrt{ \det \Im \Omega} 
\end{equation}
for all $\Omega \in \mathbb{H}_g$. The Hodge metrics $\|\cdot\|_{\Hdg}$ on $\tilde{\ee}$ resp.\ $\tilde{\ll}$ descend to give metrics, that we also denote by $\|\cdot\|_{\Hdg}$, on the bundles $\ee$ resp.\ $\ll$ on $\aa_g(\cc)$. Explicitly, let $(A,a) \in \aa_g(\cc)$ be a complex principally polarized abelian variety of dimension~$g$, then we have the identity
\begin{equation} \label{def_norm}   \|s\|_\Hdg(A,a) = (2\pi)^{gh} \cdot |\tilde{s}| (\Omega) \cdot (\det \Im \Omega)^{h/2}  
\end{equation}
for $s \in  \ss_{g,h}(\cc) $ corresponding to $\tilde{s} \in \mathcal{R}_{g,h}$. Here $\Omega$ is any element of $\mathbb{H}_g$ satisfying $u(\Omega)=(A,a)$.

Assume now that $g \geq 2$, and denote by $\mm_g$ the moduli stack of smooth proper curves of genus $g$. The Torelli map $t \colon \mm_g \to \aa_g$ gives rise to the bundles $t^*\ee$ and $t^*\ll$ on $\mm_g$. Let $\pi \colon \mathcal{C}_g \to \mm_g$ denote the universal curve of genus~$g$, and denote by $\Omega_{\mathcal{C}_g/\mm_g}$ its sheaf of relative $1$-forms. Then we have locally free sheaves $\ee_\pi = \pi_* \Omega_{\mathcal{C}_g/\mm_g}$ and $\ll_\pi = \det \ee_\pi$ on $\mm_g$, and natural identifications $\ee_\pi \isom t^*\ee$ and $\ll_\pi \isom t^*\ll$. Kodaira-Spencer deformation theory gives a canonical isomorphism 
\begin{equation} \label{KS-curves} \pi_*   \Omega_{\mathcal{C}_g/\mm_g}^{\otimes 2} \isom \Omega_{\mm_g/\zz} 
\end{equation}
of locally free sheaves on $\mm_g$.

Over $\cc$, the pullback of the Hodge metric $\|\cdot\|_{\Hdg}$ to $\ll_\pi$ coincides with the metric derived from the inner product (\ref{defhodgemetric}) introduced before.  

Let  $\overline{\mm}_g \supset \mm_g$ denote the moduli stack of stable curves of genus $g$, and consider the universal stable curve $\bar{\pi} \colon \overline{\mathcal{C}}_g \to \overline{\mathcal{M}}_g$. Let $\omega_{\overline{\mathcal{C}}_g/\overline{\mathcal{M}}_g}$ be the relative dualizing sheaf of $\bar{\pi}$, and put $\ee_{\bar{\pi}}= \bar{\pi}_*\omega_{\overline{\mathcal{C}}_g/\overline{\mathcal{M}}_g}$ and $\ll_{\bar{\pi}}= \det \ee_{\bar{\pi}}$. Then $\ee_{\bar{\pi}}$ resp.\ $\ll_{\bar{\pi}}$ are natural extensions of $\ee_\pi$ resp.\ $\ll_\pi$ over $\overline{\mm}_g$. When $S$ is a scheme or analytic space and $f \colon \xx \to S$ is a stable curve of genus $g$, we usually denote by $\ee_f = f_* \omega_{\xx/S}$ and $\ll_f = \det f_* \omega_{\xx/S}$ the sheaves on $S$ induced from $\ee_{\bar{\pi}}$ and $\ll_{\bar{\pi}}$ by the classifying map $J \colon S \to \overline{\mm}_g$ associated to $f$.

\begin{lem} \label{asympt_hdg} Let $f \colon \xx \to \DD$ be a stable curve of genus $g \geq 2$ over the open unit disk $\DD$. Assume that $f$ is smooth over $\DD^*$. Let $s $ be a Siegel modular form over $\DD^*$ of degree~$g$ and weight~$h$. Let $\Gamma \subset \Sp(2g,\zz)$ denote the image of the monodromy representation $\rho \colon \pi_1(\DD^*) \to \Sp(2g,\zz)$ induced by $f$, and let  $\Omega \colon \DD^* \to \mathbb{H}_g/\Gamma$ denote the induced period map. Then the frame $ \Omega^*(\d z_1 \wedge \ldots \wedge \d z_g) $ of $\ll_f|_{\DD^*}$ extends as a frame of $\ll_f$ over $\DD$. Furthermore, we have the asymptotics
\[  - \log \|s\|_\Hdg \sim - \ord_0(s,\ll_f)\log|t| - \frac{h}{2}\log \det \Im \Omega(t)   \]
as $t \to 0$ in $\DD^*$. The notation $\sim$ means that the difference between left and right hand side remains bounded. The term $\log \det \Im \Omega(t) $ is of order $O(\log(-\log|t|))$ and is in fact of order $O(1)$ if the special fiber $X_0$ is a stable curve of compact type. 
\end{lem}
\begin{proof} By (\ref{norm_frame}) we have $ \log  \|  \Omega^*( \d z_1 \wedge \ldots \wedge \d z_g) \|_\Hdg(t) =\frac{1}{2} \log \det \Im \Omega(t)  $ for all $t \in \DD^*$. By the Nilpotent Orbit Theorem there exists an element $c \in \zz_{\geq 0}$ such that $ \det \Im \Omega(t) \sim - c \log |t| $
as $t \to 0$. We conclude that $\Omega^*( \d z_1 \wedge \ldots \wedge \d z_g) $ extends as a frame of Mumford's canonical extension \cite{hi} of $\ll_f|_{\DD^*}$ over $\DD$. By \cite[p.~225]{fc} this canonical extension is equal to $\ll_f$. We thus obtain the first assertion. Also we obtain the equality $\ord_0(s,\ll_f)=\ord_0(\tilde{s})$, which then leads to the asymptotic $-\log |\tilde{s}| \sim -\ord_0(s,\ll_f) \log|t|$ as $t \to 0$. Combining with  (\ref{def_norm})  we find the stated asymptotics for $-\log \|s\|_\Hdg$. The element $c \in \zz_{\geq 0}$ vanishes if the special fiber $X_0$ is a stable curve of compact type. This proves the last assertion.
\end{proof}

\section{Asymptotics of the biextension metric} \label{sec:biext}

In this section we continue the spirit of the asymptotic analysis from Lemma \ref{asympt_hdg} by replacing the Hodge metric $\|\cdot\|_{\Hdg}$ with the biextension metric $\|\cdot\|_\bb$. We recall the necessary ingredients, and finish with a specific asymptotic result due to P.\ Brosnan and G.\ Pearlstein \cite{bp}. General references for this section are \cite{hain_normal} \cite{hrar} and \cite{hain}. We continue to work in the analytic category.

Let $g \geq 2$ be an integer. Let $H$ denote the standard local system of rank $2g$ over $\mm_g$. Following R.\ Hain and D.\ Reed in \cite{hrar} we have a canonical normal function section $\nu \colon \mm_g \to J$  of the intermediate jacobian $J=J(\bigwedge^3 H/H)$ over $\mm_g$, given by the Abel-Jacobi image of a Ceresa cycle on the usual jacobian $J(H)$.  

Let $\bb$ denote the natural biextension line bundle on $J$, equipped with its natural biextension metric \cite{hain}. By pulling back along the section $\nu \colon \mm_g \to J$ we obtain a natural line bundle $ \nn=\nu^* \bb$ over $\mm_g$, equipped with the pullback metric from $\bb$. By functoriality we obtain a canonical smooth hermitian line bundle $\nn$ on the base of any family $\rho \colon \CC \to B$ of smooth complex curves of genus $g$. 

As it turns out, the underlying line bundle of $\nn$ on $\mm_g$ is isomorphic with $\ll_\pi^{\otimes 8g+4}$, where $\ll_\pi = \det \ee_\pi$ is the determinant of the Hodge bundle as before. An isomorphism $\nn \isom \ll_\pi^{\otimes 8g+4}$ is determined up to a constant depending on $g$, and by transport of structure we obtain a smooth hermitian metric $\| \cdot \|_\bb$ on $\ll_\pi$, well-defined up to a constant, that we will ignore from now on.

Following \cite{hrar} we define the real-valued function
\[ \beta = \log \left(\frac{ \| \cdot \|_\bb }{ \|\cdot \|_{\Hdg}} \right) \]
on $\mm_g$. By \cite[Theorem~1.4]{djsecond}  the equality $ \beta = (8g+4)\lambda $ holds on $\mm_g$.  Let $a \in \zz_{>0}$ and let $s$ be a non-zero rational section  of $ \ll_\pi^{\otimes (8g+4)a}$ over $\mm_g$. Consider then the quantity
\begin{equation} \label{quantity}  -\frac{1}{a}\log\|s\|_{\Hdg} - (8g+4)\lambda  = -\frac{1}{a}\log\|s\|_{\Hdg} - \beta = -\frac{1}{a}\log\|s\|_\bb  
\end{equation}
on $\mm_g$.
We would like to be able to control its asymptotic behavior in smooth families over a punctured disk $\DD^*$ degenerating into a stable curve. 

Here we discuss a set-up to study this question. Consider a base complex manifold $B$ and a stable curve  $\rho \colon \CC \to B$ of genus $g \geq 2$ smooth over an open subset $U\subset B$. We then have the canonical Hain-Reed line bundle $\nn$ on $U$, equipped with its natural metric. Assume that the boundary $D=B \setminus U$ of $U$ in $B$ is a normal crossings divisor. D.\ Lear's extension result  \cite{lear}, see also \cite[Corollary~6.4]{hain_normal} and \cite[Theorem~5.19]{pearldiff}, implies that there exists a $\qq$-line bundle $[\nn,B]$ over $B$ extending the line bundle $\nn$ on $U$ in  such a way that the metric on $\nn$ extends continuously over $[\nn,B]$ away  from the singular locus of $D$.  This property uniquely determines the $\qq$-line bundle $[\nn,B]$.

For example, assume $B=\DD$ is the open unit disk, let $D$ be the origin of $\DD$ and let $f \colon X \to \DD$ be a stable curve smooth over $\DD^*=B\setminus D$.  Let $t$ be the standard coordinate on $\DD$. The existence of the Lear extension $[\nn,\DD]$ implies that there exists a rational number $b$ such that the asymptotics
\[ - \log \|s\|_\bb \sim -b \log |t| \]
holds as $t \to 0$. Here as before the notation $\sim$ means that the difference between left and right hand side remains bounded. With Lemma \ref{asympt_hdg} and (\ref{quantity}) we then find that there exists a rational number $c$ such that
\[ \beta = (8g+4)\lambda \sim -c \log |t| - (4g+2) \log \det \Im \Omega(t) \]
as $t \to 0$. One would like to compute $c$. 

Hain and Reed have shown the following result  \cite[Theorem~1]{hrar}. If $X_0$ has one node and the total space $X$ is smooth one has that
\begin{equation} \label{hr_nonsep} \beta = (8g+4)\lambda \sim -g \log|t| - (4g+2)  \log \det \Im \Omega(t) 
\end{equation}
if  the node is ``non-separating'', and
\begin{equation} \label{hr_sep} \beta = (8g+4)\lambda \sim -4h(g-h) \log|t| 
\end{equation}
if the normalization of $X_0$ consists of two connected components of genera $h>0$ and $g-h$. Referring back to Examples \ref{loop} and \ref{segment} we observe that the leading terms in the asymptotics in these cases are controlled by the $\lambda$-invariant of the polarized dual graph of the special fiber. We expect this behavior to extend to arbitrary stable curves $X \to \DD$ smooth over $\DD^*$. 

More precisely, we should have the following. Let $f \colon X \to \DD$ be a stable curve of genus $g \geq 2$ smooth over $\DD^*$. 
Let $\overline{\Gamma}$ denote the dual graph of $X_0$ endowed with its canonical polarization. Recall that if $X_0$ has $r$ nodes the graph $\Gamma$ has $r$ designated edges with weights equal to the thicknesses $(m_1,\ldots,m_r)$ of the nodes on the total space $X$. Let $\lambda(\overline{\Gamma})$ be the $\lambda$-invariant of $\overline{\Gamma}$. In general one expects that the asymptotic
\begin{equation}  \label{unknown} (?) \qquad \lambda(X_t) \sim -\lambda(\overline{\Gamma}) \log|t| - \frac{1}{2} \log \det \Im \Omega(t) 
\end{equation}
holds as $t \to 0$. However this seems not to be known in general. 

We can characterize though when this asymptotics holds in terms of the classifying map $I \colon \DD \to B$ to the universal deformation space of $X_0$, see Proposition \ref{equivalence} below. We hope that the criterion in Proposition \ref{equivalence} will be useful to prove the asymptotic in (\ref{unknown}) in general. In the present paper, we are able to verify the criterion in a special case.

The proof of the following lemma is left to the reader.
\begin{lem} \label{lear_ext_translate} Denote by $[\nn,\DD ]$ the Lear extension of the Hain-Reed line bundle $\nn$ over $\DD$. Suppose $e \in \zz_{>0}$ is such that $[\nn,\DD]^{\otimes e}$  is a line bundle on $\DD$. Denote this line bundle by $N$. Let $s$ be a generating section of  $\nn$ over $\DD^*$ and let $k \in \zz$. The following assertions are equivalent: (a) 
the asymptotic $-e \log \|s\|_\bb \sim -k \log |t|$ holds as $t \to 0$. (b)  the section $t^{-k} \cdot s^{\otimes e} $ extends as a generating section of the line bundle $N$ over $\DD$. (c) the divisor of $s^{\otimes e}$, when viewed as a rational section of $N$, is equal to $k \cdot [0]$. 
\end{lem}
Let $\rho \colon \CC \to B$ be a stable curve with $B=\DD^d$ a polydisk,  and say $\rho$ is smooth over $U=(\DD^*)^r \times \DD^{d-r}$. Consider a holomorphic arc 
\[ \bar{\phi} \colon \DD \to B\, , \quad t \mapsto (u_1t^{m_1},\ldots,u_rt^{m_r},\bar{\phi}_{r+1},\ldots,\bar{\phi}_{d}) \]
with $m_1,\ldots,m_r$ positive integers, $u_1,\ldots,u_r$ holomorphic units, and $\bar{\phi}_{r+1},\ldots,\bar{\phi}_d$ arbitrary holomorphic functions.  Let $\phi$ be the restriction of $\bar{\phi}$ to $\DD^*$. Note that $\phi$ maps $\DD^*$ into $U$. Let $s$ be a rational section of the line bundle $\ll_\rho^{\otimes (8g+4)a}$ such that $\phi^*s$ has no zeroes or poles on $\DD^*$. Let $q(m_1,\ldots,m_r) \in \qq$ for all $m=(m_1,\ldots, m_r) \in \zz_{> 0}^r$ be determined by the asymptotic
\begin{equation}  -\frac{1}{a}\log \| s \|_\bb(\phi(t)) \sim -q(m_1,\ldots,m_r) \log |t|  
\end{equation}
as $t \to 0$ (cf.\ Lear's extension result). Pearlstein proves in \cite[Theorem~5.37]{pearldiff}  that $q$ is a rational homogeneous weight one function of $m_1,\ldots, m_r$ which extends continuously over $\rr_{\geq 0}^r$. Write $q_i = q(e_i)$ where $e_i$ is the $i$-th coordinate vector in $\rr^r$.  Let $D_i$ for $i=1,\ldots,r$ denote the divisor on $B$ given by the equation $z_i=0$. Then for a holomorphic arc $\bar{\psi} \colon \DD \to B$ intersecting $D_i$ transversally and intersecting none of the $D_j$ where $j \neq i$ we have the asymptotic
\begin{equation} \label{Lear_ext_coeff}  -\frac{1}{a}\log \| s \|_\bb(\psi(t)) \sim -q_i \log |t|  
\end{equation}
as $t \to 0$. Denote by $[\nn,B ]$ the Lear extension of the Hain-Reed line bundle $\nn$ over $B$. Applying part (c) of Lemma \ref{lear_ext_translate} we find the following.
\begin{lem} \label{lem:divisor} The $\qq$-divisor of $s$, when seen as a rational section of $[\nn,B]^{\otimes a}$, is given by the $\qq$-divisor $a\sum_{i=1}^r q_i D_i$. 
\end{lem}
Following \cite[Section~14]{hain_normal}, the \emph{height jump} for the map $\bar{\phi}$ is defined to be the rational homogeneous weight one function
\begin{equation} \label{def:jump} j(m_1,\ldots,m_r) = -q(m_1,\ldots,m_r) + \sum_{i=1}^r q_i m_i \, . 
\end{equation}
Note that the height jump is independent of the choice of $s$. It was conjectured by Hain \cite[Conjecture~14.6]{hain_normal} and proved by P.~Brosnan and G.~Pearlstein in \cite{bp} (combine \cite[Corollary~11]{bp} and \cite[Theorem~20]{bp}) and independently by J. I. Burgos Gil, D. Holmes and the author in \cite[Theorem~4.1]{bghdj} that $j\geq 0$. Note that if for some $\bar{\phi} \colon \DD \to B$ as above the height jump is strictly positive, no positive tensor power of the Hain-Reed line bundle $\nn$ on $U$ extends as a continuously metrized line bundle over $B$. 

Now assume that $\rho \colon \CC \to B$ is the universal deformation of the special fiber $X_0$ of our stable curve  $f \colon X \to \DD$. Recall \cite[Section~XI.6]{acg} that the base space $B$ is a complex manifold, carrying an action of $\Aut(X_0)$, and endowed with a canonical point $b_0$ with fiber $X_0$. Locally around the point $b_0 \in B$ the divisor of singular curves in $B$ is a normal crossings divisor. 
Hence, locally around $b_0$ the family $\CC \to B$ can be identified with a stable curve over $\DD^d$ smooth over $(\DD^*)^r \times \DD^{d-r}$, for some integers $d, r$. Then for the classifying map $I \colon \DD \to B$ one has for $i=1,\ldots,r$ that $\mult_0 I^*z_i = m_i$, where $m_1,\ldots,m_r$ are the thicknesses of the nodes of $X_0$ on $X$. We find in particular a height jump $j(m_1,\ldots,m_r) \in \qq$ associated to~$I$.

Define for a pm-graph $\overline{\Gamma}$ of genus $g \geq 1$ its \emph{slope} to be the invariant
\begin{equation} \label{def_slope}  \mu(\overline{\Gamma}) = (8g+4)\lambda(\overline{\Gamma}) - g\delta_0(\overline{\Gamma}) - 4 \sum_{h=1}^{[g/2]} h(g-h) \delta_h(\overline{\Gamma}) \, . 
\end{equation}
It was conjectured by Zhang in \cite[Conjecture~1.4.5]{zhgs} and proved by Cinkir in \cite[Theorem~2.10]{ci_eff} that for all pm-graphs $\overline{\Gamma}$ we have $\mu(\overline{\Gamma}) \geq 0$. 

\begin{lem} \label{big_asymptotic}  Let $f \colon X \to \DD$ be a stable curve of genus $g \geq 2$, smooth over $\DD^*$.
 Let $\overline{\Gamma}$ denote the pm-graph associated to~$f$ and let $j$ be the height jump (\ref{def:jump}) for the classifying map $I \colon \DD \to B$ to the universal deformation space $B$ of the special fiber $X_0$. Let $a \in \zz_{>0}$ and let $s$ be a rational section of $\ll_f^{\otimes (8g+4)a}$ such that $s$ has no zeroes or poles on $\DD^*$. Then the asymptotics
\[  \begin{split} -\frac{1}{a}  \log & \|s\|_{\Hdg}(X_t) - (8g+4)\lambda(X_t) \\ & \sim -\left(\frac{1}{a}\ord_0(s,\ll_f^{\otimes (8g+4)a} )-j-(8g+4)\lambda(\overline{\Gamma})+\mu(\overline{\Gamma})   \right) \log|t| \end{split}  \]
holds as $t \to 0$, where $\mu(\overline{\Gamma})$ is the slope of $\overline{\Gamma}$ as in (\ref{def_slope}).
\end{lem}
\begin{proof} Left and right hand side of the stated asymptotics change in the same manner upon changing the rational section $s$, and hence we may assume without loss of generality that $s$ is the pullback along $I$ of a rational section of $\ll_\rho^{\otimes (8g+4)a}$, where $\rho \colon \CC \to B$ is the universal deformation of $X_0$. 
Let $m_1,\ldots,m_r$ be the multiplicities at $0 \in \DD$ of the analytic branches through $b_0 \in B$ determined by the locus of singular curves in $B$. 
Then one has the asymptotics
\[ \begin{split} -\frac{1}{a} \log & \|s\|_{\Hdg}(X_t) - (8g+4)\lambda(X_t)  \\ & = -\frac{1}{a}\log\|s\|_\bb(X_t) \\ & \sim - q(m_1,\ldots,m_r)\log |t| \\
& = -\left(-j + \sum_{i=1}^r q_i m_i\right) \log|t| \\
& = -\left(-j + \left( \sum_{i=1}^r q_i m_i - \frac{1}{a}\ord_0 (s) \right) \right) \log|t| -\frac{1}{a}\ord_0 (s) \log|t|  
\end{split} \]
as $t \to 0$.
From Lemma \ref{lem:divisor} we obtain that the $\qq$-divisor of $s$ when seen as a rational section of $[\nn,B]^{\otimes a}$ is equal to $a\sum_{i=1}^r q_iD_i$ where $D_i$ for $i=1,\ldots,r$ denotes the divisor on $B$ given by $z_i=0$. Since $\mult_0 I^*z_i = m_i$ for $i=1,\ldots,r$ it follows that
\begin{equation} \label{useful_I} a \sum_{i=1}^r q_i m_i =\ord_0(s,I^*[\nn,B ]^{\otimes a}) \, . 
\end{equation}
By \cite[Theorem~3]{hrar} we have that
\begin{equation} \label{useful_II}
 [\nn,\overline{\mm}_g] = (8g+4)\bar{\lambda}_1  - g\delta_0 -4 \sum_{h=1}^{[g/2]} h(g-h) \delta_h 
 \end{equation}
holds in the Picard group of $\overline{\mm}_g$. Here $\bar{\lambda}_1$ denotes the class of $\ll_{\bar{\pi}}$, and $\delta_i$ for $i=0,\ldots,[g/2]$ are the classes determined by the boundary divisors $\Delta_i$ on $\overline{\mm}_g$. 
From (\ref{useful_I}) and (\ref{useful_II}) we conclude that
\[ \begin{split} \sum_{i=1}^r q_i m_i - \frac{1}{a}\ord_0 (s,\ll_f^{\otimes (8g+4)a})  & =  - g \delta_0(\overline{\Gamma}) -  4 \sum_{h=1}^{[g/2]} h(g-h)\delta_h(\overline{\Gamma}) \\ & =  -(8g+4)\lambda(\overline{\Gamma}) + \mu(\overline{\Gamma}) \, . \end{split} \]
The required asymptotics follows.
\end{proof}
We deduce the following criterion to verify whether (\ref{unknown}) holds.
\begin{prop} \label{equivalence} The following assertions are equivalent: (a) one has the asymptotics \[  \lambda(X_t) \sim -\lambda(\overline{\Gamma}) \log|t| - \frac{1}{2} \log \det \Im \Omega(t) \]
as $t \to 0$, (b) one has the asymptotics 
\[ -\frac{1}{a}\log\|s\|_{\Hdg}(X_t) - (8g+4)\lambda(X_t) \sim -\left(\frac{1}{a} \ord_0(s,\ll_f^{\otimes (8g+4)a}) - (8g+4)\lambda(\overline{\Gamma}) \right) \log|t|  \]
as $t \to 0$, (c) the height jump for the classifying map $I \colon \DD \to B$ to the universal deformation space of $X_0$ and the slope of the pm-graph associated to $X$ are equal.
\end{prop}
\begin{proof} The equivalence of (a) and (b) follows from Lemma \ref{asympt_hdg}. The equivalence of (b) and (c) follows from Lemma \ref{big_asymptotic}.
\end{proof}
Now we have the following two results, that allow us to verify condition (c) in a special case.
\begin{thm} (P.~Brosnan, G.~Pearlstein \cite{bp}) \label{brosnan-pearl} Assume that the stable curve $X_0$  consists of two smooth irreducible components, one of genus $h$, one of genus $g-h-1$, joined at two points. Then the height jump $j$ for the classifying map $I \colon \DD \to B$ to the universal deformation space of $X_0$ is equal to
\[ j= \frac{4m_1m_2}{m_1+m_2} (g-h-1)h \, . \]
Here $m_1, m_2$ are the multiplicities at $0 \in \DD$ of the pullbacks of the two analytic branches through $b_0 \in B$ determined by the locus of singular curves in $B$. 
\end{thm}
\begin{proof} This follows from the calculation done in the proof of \cite[Theorem~241]{bp}. We note that \cite[Theorem~241]{bp} is about a stable curve $C$ in $\overline{\mm}_g$ consisting of two smooth irreducible components joined in two points, and with trivial automorphism group, so that the statement of \cite[Theorem~241]{bp} does not apply immediately to our setting if $X_0$ has non-trivial automorphisms. However the calculation in the proof of \cite[Theorem~241]{bp} is carried out effectively on the universal deformation space of $C$. Under the assumption that $\Aut(C)$ is trivial, this deformation space maps locally isomorphically to  $\overline{\mm}_g$.
Now the calculation in the proof of \cite[Theorem~241]{bp} on the universal deformation space of $C$ puts no particular restrictions on $\Aut(C)$, and we conclude that the expression for the height jump in \cite[Theorem~241]{bp} is valid in our setting. 
\end{proof}
\begin{prop} \label{slope} Let $\bar{\Gamma}$ be a pm-graph of genus~$g$ consisting of two vertices of genera $h$ and $g-h-1$ and joined by two edges of weights $m_1, m_2$. Then the slope of $\overline{\Gamma}$ is equal to
\[ \mu(\overline{\Gamma}) =  \frac{4m_1m_2}{m_1+m_2} (g-h-1)h \, . \]
\end{prop}
\begin{proof} This follows directly from the definition (\ref{def_slope}), and Example \ref{lambda_for_two_gon}.
\end{proof}
We observe that the height jump in Theorem \ref{brosnan-pearl} and the slope in Proposition~\ref{slope} are equal. With Proposition~\ref{equivalence} we thus obtain the following result.
\begin{cor} \label{end_result} Assume that the stable curve $X_0$  consists of two smooth irreducible components, one of genus $h$, one of genus $g-h-1$, joined at two points. Then one has the asymptotics 
\begin{equation} \label{asymp_arch} -\frac{1}{a}\log\|s\|_{\Hdg}(X_t) - (8g+4)\lambda(X_t) \sim -\left(\frac{1}{a} \ord_0(s) - (8g+4)\lambda(\overline{\Gamma}) \right) \log|t|  
\end{equation}
and 
\begin{equation}  \lambda(X_t) \sim -\lambda(\overline{\Gamma}) \log|t| - \frac{1}{2} \log \det \Im \Omega(t) 
\end{equation}
as $t \to 0$.
\end{cor}
We will use (\ref{asymp_arch}) with $g=3$ and $h=1$ for the proof of our main result.

\section{The modular form $\chi_{18}$} \label{chi}

From now on we specialize to the case that $g=3$.
We introduce the modular form $\chi_{18}$, following \cite{ig}. For more details and properties we refer to \cite{lrz} and the references therein. On Siegel's upper half space $\HH_3$ in degree 3 we have the holomorphic function
\[ \tilde{\chi}_{18}(\Omega) = \prod_{\varepsilon \, \textrm{even}} \theta_\varepsilon(0,\Omega) \, , \]
where $ \theta_\varepsilon(0,\Omega)$ denotes the Thetanullwert with characteristic $\varepsilon$, and where the product runs over all 36 even theta characteristics in genus three. We have $\tilde{\chi}_{18} \in \mathcal{R}_{3,18}$, see \cite[pp.~850--851]{ig} and \cite[Section~1]{lrz}. We define the corresponding element
\[ \begin{split} \chi_{18} & =  \tilde{\chi}_{18}(\Omega) \left( \frac{\d \zeta_1}{\zeta_1} \wedge \frac{\d \zeta_2}{\zeta_2}\wedge \frac{\d \zeta_3}{\zeta_3} \right)^{\otimes 18} \\ & = (2\pi i)^{54} \cdot \tilde{\chi}_{18}(\Omega) \cdot \left( \d z_1 \wedge \d z_2 \wedge \d z_3 \right)^{\otimes 18}  \end{split} \]
in $\ss_{3,18}(\cc)$. The analytic modular form $\tilde{\chi}_{18}$ has a Fourier expansion as a power series in the variables $q_{ij} = \exp(2\pi i \Omega_{ij})$, with coefficients in $\zz$, and by the $q$-expansion principle, cf. \cite[p.~140]{fc}, the modular form $\chi_{18}$ is defined over $\zz$, that is, we have a unique element in $\ss_{3,18}(\zz)$ whose base change to $\cc$ is equal to $\chi_{18}$. By a slight abuse of notation we also denote this element by $\chi_{18}$. 
By \cite[Proposition~3.4]{ic} the modular form $\chi'_{18}=2^{-28} \chi_{18}$ is primitive, i.e. not zero modulo $p$ for all primes~$p$. 

We recall that one has a natural structure of reduced effective Cartier divisor on the locus $H$ of hyperelliptic curves in $\mm_3$. The following result seems to be well known.
\begin{prop} \label{divisor_chi_open} The divisor of $\chi'_{18}$ on $\mm_3$ equals $2H$.
\end{prop} 
\begin{proof} Over $\cc$ this follows from (the proof of) \cite[Theorem~1]{ts}.  Recall that $\mm_3$ is smooth over $\Spec(\zz)$ with geometrically connected fibers. The primitivity of $\chi'_{18}$ then gives the statement over $\zz$. 
\end{proof}
Let $S$ be a scheme. When $f \colon \xx \to S$ is a stable curve of genus three we can view $\chi'_{18}$ as a rational section of the line bundle $\ll_f^{\otimes 18}$ on $S$. In particular, let $k$ be a number field with ring of integers $O_k$, and let $X$ be a non-hyperelliptic genus three curve with semistable reduction over $k$. Let $f \colon \xx \to S=\Spec O_k$ denote the stable model of $X$ over $k$.  
From Proposition \ref{divisor_chi_open} we obtain that $\chi_{18}'$ is generically non-vanishing on $S$, and from (\ref{arithm_degree}) we obtain the formula
\[ 18 \deg \det f_* \bar{\omega}_{\xx/S} = \sum_{v \in M(k)_0} \ord_v (\chi'_{18}) \log Nv - \sum_{v \in M(k)_\infty} \log \| \chi'_{18} \|_{\Hdg,v}  \]
for the (non-normalized) stable Faltings height of $X$. 

Combining with Corollary \ref{zhang} we deduce the following result.
\begin{thm} \label{formula_GS_genusthree} Let $X$ be a non-hyperelliptic genus three curve with semistable reduction over the number field $k$. Let $f \colon \xx \to \Spec O_k$ denote the stable model of $X$ over $k$ and view $\chi'_{18}$ as a rational section of the line bundle $\ll_f^{\otimes 18}$. Then the height of a canonical Gross-Schoen cycle $\Delta$ on $X^3$ satisfies
\[  \begin{split} \pair{\Delta,\Delta}  & = 21\left( \sum_{v \in M(k)_0} \left( \frac{1}{18} \ord_v (\chi'_{18}) -\lambda(X_v) \right) \log Nv \right. \\ & \hspace{2cm} + \left. \sum_{v \in M(k)_\infty} \left( -\frac{1}{18}\log\|\chi'_{18}\|_{\Hdg,v} - \lambda(X_v) \right) \right) \, .  \end{split} \]
\end{thm}
We will take Theorem \ref{formula_GS_genusthree} as a starting point in our proof of Theorem \ref{main_intro}.

\section{The Horikawa index} \label{sec:horikawa}

Let $S = \Spec R$ be the spectrum of a discrete valuation ring $R$. Let $f \colon \xx \to S$ be a stable curve with generic fiber smooth and non-hyperelliptic of genus three. Denote by $v$ the closed point of $S$. As above we view $\chi'_{18}$ as a rational section of the line bundle $\ll_f^{\otimes 18}$ on $S$. Then $\chi'_{18}$ is generically non-vanishing by Proposition \ref{divisor_chi_open}. The aim of this section is to give a lower bound on the multiplicity $\ord_v (\chi'_{18})$ in terms of the reduction graph of the special fiber. The result is displayed in Corollary \ref{lower_bound_ord}. We start by writing down an expression for the divisor $\divisor(\chi'_{18})$ of $\chi'_{18}$ on the moduli stack $\overline{\mm}_3$. 

Let $S$ be a scheme and let $f \colon \xx \to S$ be a stable curve of genus three. Then on $S$ we have the locally free sheaves $ \ee_f = f_* \omega_{\xx/S} $ and $ \gg_f = f_* \omega^{\otimes 2}_{\xx/S} $ as well as a natural map
\[ \nu_f \colon \Sym^2 \ee_f \to \gg_f \, , \quad \eta_1 \cdot \eta_2 \mapsto \eta_1 \otimes \eta_2 \, , \]
functorially in $f$. The map $\nu_f$ is surjective if $f$ is smooth and nowhere hyperelliptic. Both $\Sym^2 \ee_f$ and $\gg_f$ have rank six and we thus we have a natural map of invertible sheaves
\[ \det \nu_f \colon \det  \Sym^2 \ee_f  \to \det \gg_f \, , \]
functorially in $f$. 
We may and do view $\det \nu_f$ as a global section $s_f$ of the invertible sheaf 
$ (\det \Sym^2 \ee_{f})^{\otimes -1} \otimes \det \gg_f $ 
on $S$.  It has support on the locus of hyperelliptic fibers on $S$. Standard multilinear algebra yields a canonical isomorphism
\[ \det \Sym^2 \ee_f \isom  \ll_{f}^{\otimes 4}  \]
of invertible sheaves on $S$, where $\ll_f = \det \ee_f$ as before, and this shows that we may as well view $s_f$ as a global section of the invertible sheaf 
$ \ll_{f}^{\otimes - 4} \otimes \det \gg_{f} $ on $S$.  
\begin{prop} \label{smooth} Let $\pi \colon \mathcal{C}_3 \to \mm_3$ be the universal smooth curve of genus three. The section $s_\pi$ is not identically equal to zero, and the divisor of $s_\pi$ on  $\mm_3$ is equal to the reduced hyperelliptic divisor $H$.
\end{prop}
\begin{proof} The map $\nu_\pi$ is generically an isomorphism and this shows that $s_\pi$ is not identically equal to zero. Let $\Sigma = \divisor s_\pi$. Let $t \colon \mm_3 \to \aa_3$ be the Torelli map. By \cite[Section~1.3]{mo-oo} the map $t$ is finite. Let $R$ denote the ramification divisor of $t$.  By (\ref{KS-ab-schemes}) and (\ref{KS-curves}) we have canonical isomorphisms
\[  t^* \Omega_{\aa_3/\zz} \isom \Sym^2 \ee_\pi \, , \quad \gg_\pi \isom \Omega_{\mm_3/\zz} \, ,   \]
and the map $ t^* \Omega_{\aa_3/\zz} \to \Omega_{\mm_3/\zz}$ obtained by concatenating these isomorphisms with $\nu_\pi \colon \Sym^2 \ee_\pi \to \gg_\pi$ coincides with the canonical structure map. The cokernel of the latter map is $\Omega_{\mm_3/\aa_3} \cong \oo_R$ and the cokernel of $\nu_\pi$ is $\oo_\Sigma$. We find that $\Sigma = R$. By \cite[Remark~1.1]{mo-oo} the map $t$ is ramified precisely along the hyperelliptic locus. As $t$ has generic degree two we find $R=H$. Combining we obtain $\Sigma = H$.
\end{proof}
Let $\bar{\pi} \colon \overline{\mathcal{C}}_3 \to \overline{\mm}_3$ be the universal stable curve of genus three. Let $K$ denote the divisor of $s$ on $\overline{\mm}_3$ and let $\Delta$ denote the divisor of singular curves on $\overline{\mm}_3$.
\begin{prop} \label{chi_and_K}
View $\chi'_{18}$ as a rational section of the line bundle $ \ll_{\bar{\pi}}^{\otimes 18}$ on $\overline{\mm}_3$. We then have the equality of effective Cartier divisors
\[ \divisor(\chi'_{18})  = 2\,K + 2\,\Delta   \]
on $\overline{\mm}_3$. 
\end{prop}
\begin{proof}  By Proposition \ref{smooth} we have $\divisor s = H$ on $\mm_3$. From Proposition~\ref{divisor_chi_open} we recall that over $\mm_3$, the modular form $\chi'_{18}$ is a global section of $\ll_\pi^{\otimes 18}$ with divisor $2H$. We deduce that $\chi'_{18} \otimes s^{\otimes -2}$ is a trivializing section of the invertible sheaf $\ll_\pi^{\otimes 26} \otimes (\det \gg_\pi)^{\otimes -2}$ over $\mm_3$. Mumford's functorial Riemann-Roch, see \cite[Th\'eor\`eme~2.1 and equation (2.1.2)]{mb},  restricted to $\mm_3$ gives a canonical isomorphism
\[ \mu \colon \det \gg_\pi \isom \ll_\pi^{\otimes 13} \] 
of invertible sheaves on $\mm_3$. Hence, the square $\mu^{\otimes 2}$ of $\mu$ is another trivializing section of the invertible sheaf  $\ll_\pi^{\otimes 26} \otimes (\det \gg_\pi)^{\otimes -2}$ over $\mm_3$. Now we recall, cf. for example \cite[Lemme~2.2.3]{mb}, that the only invertible regular functions on $\mm_3$ are $\pm 1$. This allows us to conclude that  $\chi'_{18} \otimes s^{\otimes -2}$ and $\mu^{\otimes 2}$ are equal up to a sign. By Mumford's functorial Riemann-Roch on $\overline{\mm}_3$ the isomorphism $\mu$ extends into an isomorphism
\[  \det \gg_{\bar{\pi}} \otimes \oo(\Delta) \isom  \ll_{\bar{\pi}}^{\otimes 13}  \]
of invertible sheaves over $\overline{\mm}_3$. This gives that $\chi'_{18} \otimes s^{\otimes -2}$ extends as a trivializing section of the trivial line bundle $ \ll_{\bar{\pi}}^{\otimes 26} \otimes (\det \gg_{\bar{\pi}})^{\otimes -2} \otimes \oo(-2\Delta)$ on $\overline{\mm}_3$. The required equality of effective Cartier divisors follows.
\end{proof}
Let again $S = \Spec R$ be the spectrum of a discrete valuation ring $R$. Let $f \colon \xx \to S$ be a stable curve with generic fiber smooth and non-hyperelliptic of genus three. The morphism $\nu \colon \Sym^2 \ee_f \to \gg_f $ is surjective at the generic point, hence is globally injective. Let $\QQ_f$ denote the cokernel of $\nu$. Then $\QQ_f$ is a finite length $\oo_S$-module, and we have an exact sequence of coherent sheaves on $S$ with canonical maps,
\begin{equation} \label{ses} 0 \to \Sym^2 \ee_f \to \gg_f \to \QQ_f \to 0 \, . 
\end{equation}
Let $v$ denote the closed point of $S$. Following M.\ Reid \cite{reid}, K.\ Konno \cite{ko} and K.\ Yamaki \cite{ya} \cite{yaf} we call the integer $\length_{\oo_{S}} \QQ_f$ the \emph{Horikawa index} of $f$ at $v$, notation $\mathrm{Ind}_v(f)$.  Let $\Gamma_v$ denote the metrized graph associated to the stable curve~$f$.
\begin{prop}   \label{chi_and_Ind} View $\chi'_{18}$ as a rational section of the line bundle $\ll_f^{\otimes 18}$ on $S$. Then the equality
\[  \ord_v (\chi'_{18})  = 2 \, \mathrm{Ind}_v(f)  + 2\, \delta(\Gamma_v) \]
holds. In particular, $\chi'_{18}$ is a global section of $\ll_f^{\otimes 18}$.
\end{prop}
\begin{proof}  The Knudsen-Mumford determinant construction \cite{km} associates to each coherent sheaf $\ff$ on $S$ a functorial invertible sheaf $\det \ff$ on $S$, by using locally free resolutions. From the locally free resolution (\ref{ses}) of $\QQ_f$ we obtain a canonical isomorphism $(\det \Sym^2 \ee_f)^{\otimes -1} \otimes \det \gg_f \isom \det \QQ_f$ of invertible sheaves, and we find that $s=\det \nu$ can be viewed as a canonical non-zero global section of $\det \QQ_f$. By Proposition \ref{chi_and_K} its divisor $K$ satisfies the relation  $\divisor(\chi'_{18})  = 2\,K + 2\,\Delta$ in $\Div(S)$. This gives the identity $ \ord_v (\chi'_{18})  = 2 \ord_v(s) + 2\,\delta(\Gamma_v)$. We are thus left to prove that  $ \mathrm{Ind}_v(f)=\ord_v(s)$. By the structure theorem for finitely generated $R$-modules we can find effective Cartier divisors $K_i$ on $S$ uniquely determined by $\QQ_f$ together with a decomposition $\QQ_f = \bigoplus_i \oo/\oo(-K_i)$ of $\QQ_f$ as a direct sum of cyclic modules. The exact sequences 
\[ 0 \to \oo(-K_i) \to \oo \to \oo/\oo(-K_i) \to 0 \]
show that we have canonical identifications $ \det ( \oo/\oo(-K_i) ) = \oo(K_i)$ and by iteration we find a string of canonical identifications
\[ \det \QQ_f = \bigotimes_i \det ( \oo/\oo(-K_i) ) = \bigotimes_i \oo(K_i) = \oo(\sum_i K_i) \, . \] 
Upon noting that $\det \QQ_f = \oo(K)$ canonically we obtain the equality $K=\sum_i K_i$ of effective Cartier divisors on $S$. This implies
\[ \mathrm{Ind}_v(f)  = \sum_i \length ( \oo / \oo(-K_i) ) = \sum_i \mult_v K_i = \mult_v K = \ord_v s \, ,  \]
which is what we needed to show.
\end{proof}
Let $\overline{\Gamma}=(\Gamma,\mathbf{q})$ be a pm-graph of genus three. From \cite{ya} \cite{yaf} we recall the notion of a pair of edges of $h$-type on $\overline{\Gamma}$, and the definition of an invariant  $h(\overline{\Gamma})$. We call a vertex $v \in V(\overline{\Gamma})$  \emph{eliminable} if $v$ has valence two and satisfies $\mathbf{q}(v)=0$. We assume that $\overline{\Gamma}$ has no eliminable vertices. Then for $e, e'$ two distinct edges of $\overline{\Gamma}$ we denote by $\overline{\Gamma}^{ \{e,e'\}  }$ the contraction of all edges except $e, e'$ on $\overline{\Gamma}$. The pair $\{e,e'\}$ is called a \emph{pair of $h$-type} on $\overline{\Gamma}$ if  $\overline{\Gamma}^{ \{e,e'\}  }$ is an irreducible graph with precisely two vertices, and the induced polarization $\mathbf{q}^{ \{e,e'\}  }$ takes value~$1$ on both vertices. It can be shown \cite[Lemma~2.1]{ya} that $\overline{\Gamma}$ has at most one pair of edges of $h$-type. 

Now, if $\overline{\Gamma}$ has a pair $\{e, e' \}$ of edges of $h$-type, then we set $h(\overline{\Gamma}) = \min \{ m_1, m_2 \}$, where $m_1, m_2$ are the weights of the edges $e, e'$. If $\overline{\Gamma}$ has no pair of edges of $h$-type, then we set $h(\overline{\Gamma}) =0$. 

We continue to work with  the spectrum $S=\Spec R$ of a discrete valuation ring $R$ and a stable curve $f \colon \xx \to S$ of genus three, whose generic fiber is smooth and non-hyperelliptic.
Let $\overline{\Gamma}_v$ denote the pm-graph associated to $f$. Note that $\overline{\Gamma}_v$ has no eliminable vertices. Let $h(\overline{\Gamma}_v)$ be its $h$-invariant as above.   Let $e, e'$ be nodes of $\xx_v$. It is easy to see that the corresponding pair $\{ e,e' \}$ of edges in $\overline{\Gamma}_v$ is a pair of edges of $h$-type if and only if both $e, e'$ are of type $0$, and the partial normalization of $\xx_v$ at $\{e ,e'\}$ has exactly two connected components, both of genus one. 
\begin{prop} \label{yamaki_lower_bound} The inequality 
\[  \mathrm{Ind}_v(f) \geq h(\overline{\Gamma}_v) + 2 \,\delta_1(\overline{\Gamma}_v)   \]
holds.
\end{prop}
\begin{proof} This is \cite[Proposition~3.7]{yaf}.
\end{proof}
Combining Propositions \ref{chi_and_Ind} and \ref{yamaki_lower_bound} we find
\begin{cor} \label{lower_bound_ord} The inequality
\[  \ord_v (\chi'_{18})  \geq 2 \, h(\overline{\Gamma}_v)  + 2\, \delta_0(\overline{\Gamma}_v) + 6 \,\delta_1(\overline{\Gamma}_v) \]
holds.
\end{cor}
We saw in Section \ref{chi} that one has a natural structure of reduced effective Cartier divisor on the hyperelliptic locus $H$ in $\mm_3$. Let $\overline{H}$ be the closure of $H$ in $\overline{\mm}_3$. Then as $\overline{\mm}_3$ is smooth over $\Spec \zz$ (see \cite[Theorem~2.7]{kn}),  one has a natural structure of reduced effective Cartier divisor on $\overline{H}$. 
Not surprisingly, the Horikawa index at $v$ can be directly expressed in terms of the multiplicity of $\overline{H}$ at $v$. We do not need the next result, but we would like to mention it for completeness. 
\begin{prop}  \label{mult_and_Ind}  Let $\overline{H} $ denote the closure of the hyperelliptic locus in $\overline{\mm}_3$ as above. Then the identity
\[ \mathrm{Ind}_v(f) = \mult_v \overline{H} + 2 \, \delta_1(\overline{\Gamma}_v)    \]
holds.
\end{prop}
\begin{proof} We only give a sketch of the proof. In view of Proposition \ref{chi_and_Ind} it suffices to prove the following statement. View $\chi'_{18}$ as a rational section of the line bundle $ \ll_{\bar{\pi}}^{\otimes 18}$ on $\overline{\mm}_3$. Then the divisor $\divisor(\chi'_{18})$ of $\chi'_{18}$ satisfies the relation
\begin{equation} \label{divisor_chi} \divisor(\chi'_{18}) = 2 \,\overline{H} + 2\, \Delta_0 + 6 \,\Delta_1 \,  . 
\end{equation}
By \cite[Theorem~2.7]{kn} we have that $\overline{\mm}_3$ is proper and smooth over $\Spec \zz$ with geometrically connected fibers. The primitivity of $\chi'_{18}$ then ensures that the support of $\divisor(\chi'_{18})$ contains no vertical fibers. We are left with verifying identity (\ref{divisor_chi}) over $\cc$. From Proposition \ref{divisor_chi_open} we obtain that the divisor of $\chi'_{18}$ on $\mm_3$ equals $2H$.  To obtain the multiplicities of $\chi'_{18}$ along $\Delta_0$, $\Delta_1$ it suffices to compute the multiplicities of  $\tilde{\chi}_{18}$ along these two divisors, as $ \d z_1 \wedge \d z_2 \wedge \d z_3 $ is  trivializing in $ \ll_{\bar{\pi}}$ by Lemma \ref{asympt_hdg}. The multiplicities of $\tilde{\chi}_{18}$ along $\Delta_0$, $\Delta_1$ can be computed by writing down explicitly the Fourier expansion of $\tilde{\chi}_{18}$, see for example \cite[p.~852]{ig} for the multiplicity along $\Delta_1$.
\end{proof}

\begin{remark} Let $\bar{\lambda}_1$ denote the class of $ \ll_{\bar{\pi}}$ in $\Pic(\overline{\mm}_3)$, and let $\delta_0$, $\delta_1$ denote the classes in $\Pic(\overline{\mm}_3)$ associated to the boundary divisors $\Delta_0$, $\Delta_1$. Relation (\ref{divisor_chi}) yields, as $\Pic(\overline{\mm}_3)$ is torsion-free \cite{ac}, the relation
\[ [\overline{H}] = 9 \, \bar{\lambda}_1 - \delta_0 - 3\, \delta_1 \]
in $\Pic(\overline{\mm}_3)$. This relation is well known, cf.\ \cite{hm} for instance. 
\end{remark}

\begin{remark} \label{mult_and_h} Combining Propositions \ref{yamaki_lower_bound} and  \ref{mult_and_Ind} we obtain the inequality
\[ \mult_v \overline{H} \geq h(\overline{\Gamma}_v) \, .  \]
In particular, if the special fiber $\xx_v$ has a pair of nodes of $h$-type, then $\xx_v$ is hyperelliptic.
\end{remark}

\section{Proof of Theorem \ref{main_intro}} \label{sec:proof_main}

We can now combine all previous results in order to prove Theorem \ref{main_intro}.

Following J.\ Gu\`ardia in \cite{gu} we consider the non-hyperelliptic genus three curves over $\qq$ given by the affine equation
\[ C_n \colon \, y^4 = x^4-(4n-2)x^2 +1 \, ,\]
where  $n \in \zz$, $n \neq 0, 1$. Our aim is to prove the following result, which directly implies Theorem \ref{main_intro}.
\begin{thm} \label{main} Consider the sequence of curves $C_n$ with $n \in \zz_{>0}$, $n \equiv 2\, ( \bmod\, 3)$ and $n \not \equiv 0, 1 \, (\bmod \, 2^5)$.  Then the height of a canonical Gross-Schoen cycle on $C_n^3$  tends to infinity as $n \to \infty$.
\end{thm}
Let 
\[ a = \sqrt{n} + \sqrt{n-1} \, , \quad \mu = i \sqrt[4]{4n(n-1)} \, , \quad \zeta = (1+i)/\sqrt{2} \, , \quad k_0 = \qq(a,\zeta \mu) \, . \]
Let $\alpha = \sqrt[3]{ 18 - 6 \sqrt{3}}$, and let $\alpha_n$ be a root of the equation
\[ t^3(t^3- 24)^3 -2^8 \frac{ (n^2-n+1)^3}{n^2(n-1)^2 } (t^3-27) = 0 \, .    \]
Put $k_n = k_0(\alpha,\alpha_n)$. By \cite[Section~3]{gu}, for $n$ as in Theorem \ref{main} the curve $C_n$ acquires semistable reduction over $k_n$. 

Let $f \colon \xx_n \to \Spec O_n$ denote the stable model of $C_n$ over the ring of integers $O_n$ of $k_n$. View $\chi'_{18}$ as a rational section of the invertible sheaf $\ll_f^{\otimes 18}$ on $\Spec O_n$. By Theorem \ref{formula_GS_genusthree} we obtain for the height $\pair{\Delta_n,\Delta_n}$ of a canonical Gross-Schoen cycle $\Delta_n$ on $C_n^3$  that
\[  \begin{split} \pair{\Delta_n,\Delta_n} & = \frac{21}{[k_n\colon\qq]}\left( \sum_{v \in M(k_n)_0} \left( \frac{1}{18} \ord_v (\chi'_{18})-\lambda(C_{n,v}) \right) \log Nv \right) \\ &  \hspace{2.5cm}  -\frac{1}{18}\log\|\chi'_{18}\|_{\Hdg}(C_n) - \lambda(C_n)  \, . \end{split}  \]
In Section \ref{sec:proof_finite} we will show the following result. 
\begin{thm} \label{finite_places} For $n$ as in Theorem \ref{main} and for $v \in M(k_n)_0$ we have that the local non-archimedean contribution $(\frac{1}{18}\ord_v (\chi'_{18}) -\lambda(C_{n,v}))\log Nv$  is non-negative.
\end{thm}
\begin{cor}  \label{lower_bound_height}  For $n$ as in Theorem \ref{main}  the inequality
\[ \pair{\Delta_n,\Delta_n} \geq  -\frac{1}{18}\log\|\chi'_{18}\|_{\Hdg}(C_n) - \lambda(C_n)  \]
holds.
\end{cor}
In Section \ref{sec:archimedean} we will show
\begin{thm} \label{asympt_origami} There exists a positive rational number $c$ such that for $n \to \infty$ the asymptotics
\[  -\frac{1}{18}\log\|\chi'_{18}\|_{\Hdg}(C_n) - \lambda(C_n) \sim c \log n  \]
holds. 
\end{thm}
Combining Corollary \ref{lower_bound_height} and Theorem \ref{asympt_origami} one finds Theorem \ref{main}.

\section{Proof of Theorem \ref{finite_places}} \label{sec:proof_finite}

Assume that $n$ is as in Theorem \ref{main}. The reduction types of $C_n$ at all $v \in M(k_n)_0$ are given in \cite{gu}. 
At a prime of $O_n$ not dividing $n(n-1)$ the curve $C_n$ has good reduction, by \cite[Proposition~3.1]{gu}. At a prime of $O_n$ dividing $2$ we have by \cite[Theorem~7.4]{gu} that the special fiber consists of a smooth genus zero component, with three disjoint elliptic curves attached to it. The dual graph of the special fiber is thus a polarized tree in this case. Finally, at an odd prime of $O_n$ dividing $n(n-1)$ the special fiber is the union of two elliptic curves meeting in two distinct points, by \cite[Theorem~5.3]{gu}. Hence in this case the polarized dual graph of the special fiber consists of two vertices of genus one, joined by two edges. 

We now analyze each of these various cases.  Let $S = \Spec R$ be the spectrum of a discrete valuation ring $R$. Let $f \colon \xx \to S$ be a stable curve with generic fiber smooth and non-hyperelliptic of genus three. Let $\overline{\Gamma}$ denote the polarized dual graph associated to $f$. 

Theorem \ref{finite_places} is proved by the following three lemmas.
\begin{lem} Assume that $f \colon \xx \to S$ is smooth. Then $\frac{1}{18} \ord_v (\chi'_{18}) -\lambda(\overline{\Gamma}) \geq 0$.
\end{lem}
\begin{proof} We have $\ord_v (\chi'_{18})= 2 \mult_v H \geq 0$ by Proposition \ref{divisor_chi_open}, and $\lambda(\overline{\Gamma})=0$.
\end{proof}
\begin{lem} Assume that $\Gamma$ is a tree. Then $\frac{1}{18} \ord_v (\chi'_{18}) -\lambda(\overline{\Gamma}) \geq \frac{1}{21} \delta(\Gamma) $. 
\end{lem}
\begin{proof} By Example \ref{tree} we have $\lambda(\overline{\Gamma}) = \frac{2}{7}\delta(\Gamma)$  and since $h(\overline{\Gamma})=\delta_0(\overline{\Gamma})=0$ and $\delta_1(\overline{\Gamma})=\delta(\Gamma)$ in this case we obtain by Corollary \ref{lower_bound_ord}
\[ \frac{1}{18} \ord_v (\chi'_{18}) -\lambda(\overline{\Gamma}) =   \frac{1}{18} \ord_v (\chi'_{18}) -  \frac{2}{7}\delta(\Gamma) \geq  \left( \frac{1}{3}  - \frac{2}{7} \right) \delta(\Gamma) = \frac{1}{21} \delta(\Gamma)  \, . \]
This proves the lemma. 
\end{proof}
\begin{lem} \label{glue_elliptic} Assume that the special fiber of $f$ is the union of two elliptic curves meeting in two distinct points. Let $m_1, m_2 \in \zz_{>0}$ be the thicknesses on $\xx$ of the two singular points of the special fiber. Then we have
\[ \frac{1}{18} \ord_v (\chi'_{18}) -\lambda(\overline{\Gamma}) \geq  \frac{1}{9\cdot28}(m_1+m_2) +\frac{1}{9} \min \{m_1,m_2 \} - \frac{1}{7} \frac{m_1m_2}{m_1+m_2} \, . \]
The expression on the right hand side is strictly positive.
\end{lem}
\begin{proof}

We have $\delta_1(\overline{\Gamma})=0$, $\delta_0(\overline{\Gamma})=m_1+m_2$ and $h(\overline{\Gamma})=\min \{m_1,m_2\}$.  By Corollary \ref{lower_bound_ord} we have
\[ \ord_v (\chi'_{18}) \geq 2\min \{m_1,m_2\} + 2(m_1+m_2) \, . \]
Example \ref{lambda_for_two_gon} gives
\[ \lambda(\overline{\Gamma}) = \frac{3}{28}(m_1+m_2) + \frac{1}{7} \frac{m_1m_2}{m_1+m_2} \, . \]
Combining we find the required inequality. 
Assume that $m_2 \geq m_1$ and write $m_1=m$, $m_2=m+n$, with $m >0$, $n \geq 0$. The positivity of the right hand is then readily seen to be equivalent to the positivity of the quadratic form $24m^2 - 4mn + n^2$. The latter form is positive definite, and its positivity for all $m >0$, $n \geq 0$ follows.
\end{proof}
\begin{remark}  In general, by Corollary \ref{lower_bound_ord} one would like to be able to check for a given pm-graph $\overline{\Gamma}$ of genus three whether the inequality
\[ (?) \qquad \frac{1}{9} h(\overline{\Gamma})  + \frac{1}{9} \delta_0(\overline{\Gamma}) + \frac{1}{3} \delta_1(\overline{\Gamma}) -\lambda(\overline{\Gamma}) \geq 0 \]
is satisfied. Z.\ Cinkir gives in \cite{ci_adm} a list of all types of pm-graphs in genus three and displays $\lambda(\overline{\Gamma})$ for each of them. Thus inequality (?) can be checked using \cite{ci_adm} for any given $\overline{\Gamma}$. In \cite{ya} an invariant $\Phi(\overline{\Gamma})$ is introduced, and it turns out that
\[  \frac{1}{9} h(\overline{\Gamma})  + \frac{1}{9} \delta_0(\overline{\Gamma}) + \frac{1}{3} \delta_1(\overline{\Gamma}) -  \lambda(\overline{\Gamma}) = \frac{1}{12}\Phi(\overline{\Gamma}_0) + \frac{1}{21}\delta_1(\overline{\Gamma}) \, , \]
where $\overline{\Gamma}_0$ is the pm-graph obtained from $\overline{\Gamma}$ by contracting all edges of type~$1$.
In \cite[Theorem~2.7]{ya} sufficient conditions are given for non-negativity of $\Phi(\overline{\Gamma})$ for $\overline{\Gamma}$ with no edges of type~$1$. As is noted in \cite[Remark~2.9]{ya}, there exist $\overline{\Gamma}$ of genus three with no edges of type~$1$ for which $\Phi(\overline{\Gamma})$ is negative!
\end{remark}

\section{Proof of Theorem \ref{asympt_origami}} \label{sec:archimedean}

Specializing  (\ref{asymp_arch}) from Corollary \ref{end_result} to the case $g=3$, $h=1$ and working with a suitable power of $\chi'_{18}$ we obtain  the following.
\begin{thm} \label{asympt_general_union} Let $f \colon X \to \DD$ be a stable curve of genus three, smooth over $\DD^*$, whose generic fiber is non-hyperelliptic and whose special fiber $X_0$ consists of the union of two elliptic curves, joined at two points. Let $\overline{\Gamma}$ be the pm-graph associated to $f$. View $\chi'_{18}$ as a rational section of the line bundle $\ll_f^{\otimes 18}$ on $\DD$. Then one has the asymptotics
\[  -\frac{1}{18}\log\|\chi'_{18}\|_{\Hdg}(X_t) - \lambda(X_t) \sim  - \left(  \frac{1}{18} \ord_0(\chi'_{18})- \lambda(\overline{\Gamma})  \right)  \log |t|  \]
as $t \to 0$. 
\end{thm}
We deduce Theorem \ref{asympt_origami} from the latter result. Take $n \in \cc \setminus \{0,1\}$. It is readily checked that the curve $C_n \colon y^4 = x^4-(4n-2)x^2+1$ is isomorphic over $\cc$ with the curve
\[ D_\kappa \colon  \, y^4 = x(x-1)(x-\kappa) \, , \]
where $\kappa=1/n$. Indeed, the four roots of  $x^4-(4n-2)x^2+1$ are given by $\pm \sqrt{2n-1 \pm 2\sqrt{n^2-n}}$, and for a suitable ordering of these four roots, the associated cross ratio is $1/n$.

As $\kappa \to 0$ the curves $D_\kappa$ degenerate into the tacnodal curve $y^4=x^2(x-1)$. 
By \cite[Proposition~8]{hs} the stable reduction  $X \to \DD$ of the family $D_\kappa$ at $\kappa=0$ has as special fiber the union of two copies of the elliptic curve $E_{-1} $ given by the equation $y^2=x^3-x$, joined at two points. By Lemma \ref{glue_elliptic} the rational number $b= \frac{1}{18} \ord_0(\chi'_{18})- \lambda(\overline{\Gamma})$ associated to the stable family $X \to \DD$ is strictly positive. Take $d$ a positive integer such that the family $D_\kappa \cong C_n$ has semistable reduction near $\kappa=0$ after a ramified base change of degree $d$ (we can take $d=2$ but this is not important). Putting $t = \sqrt[d]{\kappa}=\sqrt[d]{1/n}$ we deduce from Corollary \ref{asympt_general_union} that
\[ -\frac{1}{18}\log\|\chi'_{18}\|_{\Hdg}(C_n) - \lambda(C_n) \sim -b \log|t| = \frac{b}{d}  \log n \]
as $n \to \infty$. Theorem \ref{asympt_origami} follows.

\vspace{0.5cm}

\noindent Address of the author:\\ \\
Mathematical Institute  \\
Leiden University \\
PO Box 9512  \\
2300 RA Leiden  \\
The Netherlands  \\ \\
Email: \verb+rdejong@math.leidenuniv.nl+

\end{document}